\documentclass[11pt]{article}

\usepackage[margin=1in]{geometry}

\usepackage{lmodern}
\usepackage[T1]{fontenc}

\usepackage[english]{babel} 
\usepackage[utf8]{inputenc}

\usepackage{mathrsfs} 
\usepackage{amsfonts}
\usepackage{marvosym} 
\usepackage{amsmath,amssymb,amsthm}

\usepackage{bm}

\usepackage[dvipsnames, svgnames]{xcolor}
\definecolor{vert}{RGB}{13,75,47}
\definecolor{gris}{RGB}{128,128,128}
\definecolor{bleu}{RGB}{0,50,150}
\definecolor{rouge}{RGB}{162,30,2}
\usepackage[colorlinks = true,linkcolor=vert,citecolor=rouge]{hyperref}

\usepackage[nameinlink]{cleveref} 
\crefname{equation}{}{}
\usepackage{bbm}

\usepackage{epsfig}

\usepackage{tikz-cd} 

\usepackage{stmaryrd}
\usepackage{manfnt}
\usepackage{multicol}
\usepackage{array}
\usepackage{multirow}

\usepackage{fancyhdr} 
\usepackage{fancybox} 

\usepackage{calligra}
\title{Deligne $1$-motives with torsion and étale motives}
\author{Raphaël Ruimy}
\date{}

\usepackage{tikz}
\usetikzlibrary{calc, intersections}

\usepackage{adjustbox}

\theoremstyle{plain}
\newtheorem{theorem}{Theorem}[section]
\newtheorem{proposition}[theorem]{Proposition}

\newtheorem*{theorem*}{Theorem}
\newtheorem*{mainthm}{Main Theorem}
\newtheorem*{prop*}{Proposition}

\newtheorem{hyp}[theorem]{Hypothesis}
\newtheorem{lemma}[theorem]{Lemma}
\newtheorem{corollary}[theorem]{Corollary} 

\theoremstyle{definition}
\newtheorem{definition}[theorem]{Definition}
\newtheorem*{definition*}{Definition}
\crefname{hyp}{Hypothesis}{Hypotheses}

\newtheorem{constr}[theorem]{Construction}

\theoremstyle{remark}

\newtheorem{rem}[theorem]{Remark}
\newtheorem*{rem*}{Remark}

\numberwithin{equation}{section}

\newcommand{\Q}{\mathbb{Q}}
\newcommand{\Z}{\mathbb{Z}}

\newcommand{\Spec}{\mathrm{Spec}}
\newcommand{\A}{\mathbb{A}}

\newcommand{\Gal}{\mathrm{Gal}}

\newcommand{\D}{\mathrm{D}}

\newcommand{\acal}{\mathcal{A}}

\newcommand{\ecal}{\mathcal{E}}

\newcommand{\ccal}{\mathcal{C}}

\newcommand{\mc}{\mathcal}

\newcommand{\mb}{\mathbb}
\newcommand{\proet}{\mathrm{pro\acute{e}t}}

\newcommand{\Perf}{\mathrm{Perf}}

\newcommand{\colim}{\mathrm{colim}}

\newcommand{\Hom}{\mathrm{Hom}}

\newcommand{\coker}{\mathrm{coker}}
\DeclareMathOperator{\sHom}{\mathscr{H}\text{\kern -3pt {\calligra\large om}}\,}

\newcommand{\Map}{\mathrm{Map}}
\newcommand{\map}{\mathrm{map}}

\newcommand{\pp}{\mathfrak{p}}

\newcommand{\et}{\mathrm{\acute{e}t}}\newcommand{\DM}{\mathrm{DM}}

\newcommand{\Sh}{\mathrm{Sh}}

\newcommand{\Sch}{\mathrm{Sch}}

\newcommand{\heart}{\heartsuit}

\newcommand{\rar}{\to}

\newcommand{\Sm}{\mathrm{Sm}}
\newcommand{\sm}{\mathrm{sm}}
\newcommand{\AAA}{\mathbb{A}^1}

\newcommand{\Loc}{\mathrm{Loc}}
\newcommand{\rmm}{\mathrm{M}}
\newcommand{\rmc}{\mathrm{C}}
\usepackage{appendix}
\makeatother

\begin{document}

\maketitle
\tableofcontents

\begin{abstract}
    We construct the motivic t-structure on $1$-motives with integral coefficients over a scheme of characteristic zero or a Dedekind scheme. When we invert the residue characteristic exponents of the base, this t-structure induces a t-structure on the category of smooth $1$-motives whose heart is the category of Deligne $1$-motives with torsion which we prove to be abelian. This relies on the fact that over a normal scheme, Deligne $1$-motives with torsion are determined by their fiber on the generic point and on an explicit description of good reduction Deligne $1$-motives with torsion. 
\end{abstract}

\section*{Introduction}
Connecting the categories of mixed étale motives from \cite{thesevoe,ayo14,em} to Grothendieck's conjectural theory of motives would require defining a motivic t-structure on étale motives. Over a field, one of the most substantial advances on this problem was the definition of t-structure on the subcategory of $1$-motives that we can link to Deligne's category of $1$-motives from \cite{hodgeIII}. 
\begin{theorem*}(\cite{orgo,ayo11,abv,bvk})
    Let $k$ be a field of residue characteristic exponent $p$. There is a t-structure on the category $\DM_{\et}^1(k,\Z)$ of (geometric) Voevodsky étale $1$-motives whose heart is the abelian category of Deligne $1$-motives with torsion $\rmm_1^\D(k,\Z[1/p])$. Furthermore, we have an equivalence 
    \[\D^b(\rmm_1^\D(k,\Z[1/p]))\to \DM^1_{\et}(k,\Z).\]
\end{theorem*}

This problem can also be considered over an arbitrary base scheme, but so far, only the case of motives with rational coefficients was solved.

\begin{theorem*}(\cite{plh,plh2,vaish}) Let $S$ be a noetherian excellent finite dimensional scheme allowing resolution of singularities by alterations, let $\ell$ be a prime number.
There is a t-structure on category $\DM^1_{\et}(S,\Q)$ of $1$-motives over $S$ with rational coefficients such that the $\ell$-adic realization functor
$$\rho_\ell\colon\DM^1_{\et}(S,\Q)\rar \D_\mathrm{cons}(S[1/\ell],\Q_\ell)$$
is t-exact when the derived category $\D_\mathrm{cons}(S[1/\ell],\Q_\ell)$ of constructible $\ell$-adic complexes over $S[1/\ell]$ is endowed with its ordinary t-structure.

Furthermore, there is a fully faithful functor 
\[\Phi_S\colon\rmm_1^\D(S,\Q)\to \DM^1(S,\Q)^\heart\] from the category $\rmm_1^\D(S,\Q)$ of Deligne $1$-motives with rational coefficients to the heart of the t-structure.
\end{theorem*}

We know that the existence of the motivic t-structure with rational coefficients implies the existence of the t-structure with integral coefficients (see for instance\cite{IntegralNori}) but this doesn't apply to $1$-motives directly. Furthermore, the proofs in \cite{plh,plh2,vaish} cannot be adapted to the integral setting: they rely on the fact that with rational coefficients, the motivic Picard functor $\omega^1$ from \cite{plh} preserves compact objects which allows to understand a motive from its restriction on an open subscheme and on its closed complement by means of a localization triangle. With integral coefficients, this fails completely: even the Artin truncation functor $\omega^0$ can have a very pathological behavior as it almost never preserves constructible objects (see \cite[Proposition~2.2.1]{AM2}). 

In \cite{AM1}, we proved that in the case of $0$-motives, we can in fact build up from the case of rational coefficients to produce a t-structure on the category $\DM^0_{\et}(S,\Z)$ of $0$-motives:
\begin{theorem*}(\cite{AM1}) Let $S$ be a noetherian excellent finite dimensional scheme allowing resolution of singularities by alterations and let $\ell$ be a prime number.
There is a t-structure on category $\DM^0_{\et}(S,\Z)$ of $0$-motives over $S$ such that the $\ell$-adic realization functor
$$\rho_\ell\colon\DM^0_{\et}(S,\Z)\rar \D_\mathrm{cons}(S[1/\ell],\Z_\ell)$$
is t-exact when the derived category $\D_\mathrm{cons}(S[1/\ell],\Z_\ell)$ of constructible $\ell$-adic complexes over $S[1/\ell]$ is endowed with its ordinary t-structure.
\end{theorem*}
This result relies on understanding how the candidate t-structure on Ind-$0$-motives\footnote{In this paper, this will always be read as (Ind-$0$)-motives, meaning colimits of (geometric) $0$-motives, and not as the Ind-category of $0$-motives which does not embed in étale motives in general.} interacts with the canonical t-structure on torsion étale sheaves through the rigidity theorem and on an explicit description of the subcategory $\DM^{\sm0}_{\et}(S,\Z')$ of smooth $0$-motives for $S$ regular and for $\Z'$ the localization of $\Z$ at all the residue characteristic exponents of $S$: the latter is equivalent to the category $\D_\mathrm{lisse}(S_{\et},\Z')$ of dualizable sheaves on the small étale site which has a canonical t-structure that we can compare to the t-structure on Ind-$0$-motives. 

Such an explicit description of smooth $1$-motives doesn't seem to exist at the moment. However, in \cite{haas}, Haas proved the following result:
\begin{theorem*}Let $S$ be a noetherian $\Q$-scheme of finite dimension or a Dedekind scheme, then the functor $\Phi_S$ induces an equivalence:
\[\Phi_S\colon\rmm_1^\D(S,\Q)\xrightarrow{\sim} \DM^{\sm1}_{\et}(S,\Q)\cap \DM^{\mathrm{ind}1}_{\et}(S,\Q)^\heart\]
where $\DM^{\sm1}_{\et}(S,\Q)$ is the category of dualizable $1$-motives. Furthermore, the category $\rmm_1^\D(S,\Q)$ is abelian.
\end{theorem*}
This result in fact implies that the t-structure restricts to $\DM^{\sm1}_{\et}(S,\Q)$ as we will see in the proof of \Cref{DMsm1t} below because $\DM^{\sm1}_{\et}(S,\Q)\cap \DM^{\mathrm{ind}1}_{\et}(S,\Q)^\heart$ is then closed under kernel and cokernel (the only thing to check is that the functor induced by $\Phi_S$ is exact). More recently, Haas' approach has been used to prove that Nori $0$-motive coincide with Voevodsky $0$-motives over a base scheme in \cite{Swann0mot}.

We can therefore devise a strategy to prove the existence of the motivic t-structure on $\DM^1_{\et}(S,\Z)$, with the same restrictive hypothesis as Haas. It goes as follows: first define a t-structure on Ind-$1$-motives, then prove that it restricts to smooth $1$-motives by comparing them to Deligne $1$-motives and finally mimic the proof of \cite{AM1}. Along the way, we will prove some results on Deligne $1$-motives that we think to be interesting for themselves, most notably that they form an abelian category up to inverting the characteristic exponents of $S$.

Let us now describe the content of this paper in more details. 

\subsection*{Deligne \texorpdfstring{$1$}{1}-motives with torsion}

We first define the category of Deligne $1$-motives with torsion over an arbitrary noetherian base scheme $S$ (see \Cref{defdel1mot}). To that end, we start from the category $\widetilde{\rmm}_1^\D(S)$ of \emph{effective Deligne $1$-motives}, that is complexes of the form $[L\to G]$ for $L$ a discrete group scheme and $G$ semi-abelian. We then localize at the class of maps which are quasi-isomorphisms of complexes to obtain the category $\rmm_1^\D(S)$ of Deligne $1$-motives with torsion. This category contains Deligne's category of $1$-motives from \cite{hodgeIII} (see \Cref{usualDel}). Our definition is slightly more restrictive than that of \cite{jossen} but it coincides with \cite{bvk,BVRS} over a field. The category we define is \textit{a priori} an $\infty$-category but we prove that it is in fact a $1$-category
(\Cref{ComputationHomM1D}) by proving that the localization that we consider is a Gabriel-Zisman localization. 
We end the first section by defining the category $\rmm_1^\D(S,\Lambda)$ of Deligne $1$-motives with coefficients in a flat $\Z$-algebra $\Lambda$.

\subsection*{Good reduction for Deligne \texorpdfstring{$1$}{1}-motives with torsion}
Over a connected normal scheme, we compare Deligne $1$-motives to Deligne $1$-motives over the generic point. Let us denote by $\Z'$ the ring $\Z[1/p\mid p \text{ residue characteristic exponent of }S].$ The following result generalizes \cite[Proposition~A.11]{plh}.
\begin{theorem*}(\Cref{embeddinggeneric2})
    Let $S$ be a connected normal noetherian scheme of finite dimension with generic point $\eta$ and let $\Lambda$ be a flat $\Z'$-algebra. 
Then, the restriction to the generic point 
\[\rmm_1^\D(S,\Lambda)\to \rmm_1^\D(\eta,\Lambda)\]
is fully faithful.
\end{theorem*}

It is known from \cite{bvk} that $\rmm_1^\D(\eta,\Z')$ is an abelian category. Understanding when a Deligne $1$-motive with torsion at $\eta$ extends to $S$ (that is has good reduction), is a crucial point in Haas' proof that $\rmm_1^\D(S,\Q)$ is an abelian category. We extend his \cite[Theorem~4.10]{haas} to the case of $\Z'$-coefficients:

\begin{prop*}(\Cref{bonnereduction} and \Cref{abeliancat})
Let $S$ be a connected normal scheme with generic point $\eta$ of characteristic exponent $p$ and let $\Lambda$ be a flat $\Z'$-algebra. Assume that $S$ is either a $\Q$-scheme or a Dedekind scheme. 
A Deligne $1$-motive with coefficients $\Lambda$ over $\eta$ has good reduction if and only if for any prime number $\ell\neq p$, its $\ell$-adic Tate module has good reduction on $S[1/\ell]$. 

Hence, if $\Lambda$ is a localization of $\Z'$, the category $\rmm_1^\D(S,\Lambda)$ is a Serre subcategory of $\rmm_1^\D(\eta,\Lambda)$. In particular, it is an abelian category.
\end{prop*}
This proposition relies notably on Grothendieck's theorem on extension of abelian varieties in characteristic zero \cite{Grothendiecketendable} and on the criterion of Néron Ogg and Shafarevich of \cite{Serre-Tate}.

\subsection*{Voevodsky \texorpdfstring{$1$}{1}-motives}
The goal of this section is to compare Deligne $1$-motives to Voevodsky motives. We first prove that the $\ell$-adic realization functors form a  conservative family on $1$-motives (\Cref{conservativityconj}). We then define a functor $\Phi_S\colon \rmm_1^\D(S,\Lambda)\to \DM_{\et}(S,\Lambda).$ 
We then generalize \cite[Proposition~2.15]{plh}:
\begin{theorem*}(\Cref{Delignemotare1mot})
    The essential image of $\Phi_S$ is contained in the category $\DM_{\et}^{\sm1}(S,\Lambda)$ of smooth $1$-motives.
\end{theorem*}

\subsection*{The ordinary t-structure on Ind-\texorpdfstring{$1$}{1}-motives}

We then define the motivic t-structure on Ind-$1$-motives. It is defined by imposing a set of non-positive generators given by the images of Deligne $1$-motives in Voevodsky motives. Over a field $k$, this t-structure induces a t-structure on the category $\DM_{\et}^1(k,\Z[1/p])$ of $1$-motives which is the same as the t-structure given by the equivalence with $\D^b(\rmm_1^\D(k,\Z[1/p]))$ of
 \cite{bvk} (\Cref{bvkvsord}). By showing how the t-structure interacts with torsion objects  (\Cref{torsioninduced}), we prove that the t-structure is compatible with pullbacks and with the $\ell$-adic realization on the subcategory of $1$-motives (\Cref{f^* ordinary}). In particular $\Phi_S$ sends Deligne $1$-motives with torsion to the heart of the motivic t-structure (\Cref{Deligne1motareinheart}).

\subsection*{The motivic t-structure on $1$-motives}
In the finial section, we prove that the motivic t-structure restricts to $1$-motives. The first step is the case of smooth $1$-motives over a regular base. 
\begin{theorem*}(\Cref{DMsm1t})
Let $S$ be a regular scheme of finite dimension which is either a $\Q$-scheme or a Dedekind scheme and let $\Lambda$ be localization of $\Z'$. Then, $\Phi_S$ induces an equivalence:
    \[\Phi_S\colon\rmm_1^\D(S,\Lambda)\xrightarrow{\sim}\DM^{\sm1}_{\et}(S,\Lambda)\cap \DM^{\mathrm{ind}1}_{\et}(S,\Lambda)^\heart.\]
In particular, the motivic t-structure induces a t-structure on $\DM^{\sm1}_\et(S,\Lambda)$.
\end{theorem*}
As we know that Deligne $1$-motives with torsion embed in Deligne $1$-motives with torsion over the generic point, the full faithfulness amounts to a similar result for the intersection of smooth $1$-motives with the heart which is \Cref{embeddingthm}. This lemma relies on a computation on the motivic Picard functor $\omega^1$. The idea is that for an open immersion $j$, the first few cohomology groups of $\omega^1j_*j^*(M)$ for $M$ smooth and in the heart are not badly behaved and they do not see the pathological behavior related to torsion objects which appears only in higher degrees. This is consistent with the spirit of \cite{AM2}.
On the other hand, the essential surjectivity follows from our characterization of good reduction Deligne $1$-motives. The main theorem of this paper then follows from the above result.
\begin{mainthm}
(\Cref{ordinary is trop bien})  Let $S$ be a $\Q$-scheme or a Dedekind scheme and
let $\Lambda$ localization of $\Z$. Then the motivic t-structure induces a t-structure on $\DM^1(S,\Lambda)$. This t-structure is furthermore compatible with pullback and with the $\ell$-adic realization functors.
\end{mainthm}

\section*{Notations and conventions}

In this text we will freely use the language of $\infty$-categories of \cite{htt,ha,cisinski}.
When we refer to derived categories, we always refer to the $\infty$-categorical version. If $\ccal$ is a site and $\Lambda$ is a commutative ring, we let $\Sh(\ccal,\Lambda)$ be category of $\Lambda$-module objects in $\ccal$ and $\D(\ccal,\Lambda)$ be its derived category.
We adopt the cohomological convention for t-structures (\textit{i.e} the convention of \cite[1.3.1]{bbd}).
A stable $\infty$-category endowed with a t-structure is called a t-category. If $\mc{D}$ is a t-category, we denote by $\mc{D}^\heart=\mc{D}^{\geqslant 0} \cap \mc{D}^{\leqslant 0}$ the heart of the t-structure which is an abelian category. 

All schemes are assumed to be \textbf{noetherian} and \textbf{finite dimensional}; furthermore all smooth (and étale) morphisms and all quasi-finite morphisms are implicitly assumed to be separated and of finite type. 
We let $\Sm$ be the class of smooth morphisms of schemes. For a scheme $S$, we let $S_{\et}$ (resp. $S_\proet$, resp. $\Sm_S$, resp. $\mathrm{Sch}_S$) be the category of étale (resp. weakly étale, resp. smooth, resp. arbitrary) $S$-schemes.

Let $S$ be a scheme and let $\Lambda$ be a commutative ring. The stable $\infty$-category $\DM_{\et}(S,\Lambda)$ of \emph{étale motives} is the $\mb{P}^1$-stabilization of the category $\D_{\AAA}((\Sm_S)_{\et},\Lambda)$ of $\AAA$-local objects of $\D((\Sm_S)_\et,\Lambda)$. We have a chain of adjunctions that are monoidal and compatible with pullbacks.
\[\begin{tikzcd}
	{\D(S_{\et},\Lambda)} & {\D((\Sm_S)_{\et},\Lambda)} & {\D_{\mathbb{A}^1}((\Sm_S)_{\et},\Lambda)} & {\DM_{\et}(S,\Lambda)}
	\arrow["{\rho_\sharp}", shift left, from=1-1, to=1-2]
	\arrow["{\rho^\sharp}", shift left, from=1-2, to=1-1]
	\arrow["{L_{\mathbb{A}^1}}", shift left, from=1-2, to=1-3]
	\arrow["\iota", shift left, from=1-3, to=1-2]
	\arrow["{\Sigma^\infty}", shift left, from=1-3, to=1-4]
	\arrow["{\Omega^\infty}", shift left, from=1-4, to=1-3]
\end{tikzcd}\]
we denote by $(\rho_!,\rho^!)$ the composition of these adjunctions. 

If $\Lambda$ is a topological ring, we can see it as a sheaf on the proétale site of the point which for any scheme $S$ gives rise to a sheaf $\Lambda_S$. We let $\mathrm{Loc}_S(\Lambda)$ be the category of locally constant sheaves of $\Lambda$-modules over $S_{\proet}$ with finitely presented values. It is an abelian category when $\Lambda$ is regular (as a ring) and it is equivalent to the category of locally constant sheaves of $\Lambda$-modules over $S_{\et}$ with finitely presented values when $\Lambda$ is discrete (see \cite{hrs}).

Let $i\colon Z\rar X$ be a closed immersion and $j\colon U\rar X$ be the complementary open immersion. We call localization triangles the exact triangles of functors: 
\begin{equation}\label{localization}j_!j^*\rar Id \rar i_*i^*.
\end{equation}
\begin{equation}\label{colocalization}i_!i^!\rar Id \rar j_*j^*.
\end{equation}
in $\DM_{\et}(-,\Lambda)$.
\section{Deligne \texorpdfstring{$1$}{1}-motives with torsion}
In this section, we define the category of Deligne $1$-motives with torsion over a base scheme. This category contains Deligne's category of $1$-motives from \cite[Variante~10.1.10]{hodgeIII}. Our definition is slightly more restrictive than that of \cite[Definition~1.1.3]{jossen} and coincides with \cite[Definition~C.3.1]{bvk} over a field and \cite[Section~1.1]{BVRS} over a field of characteristic $0$. The category we define is \textit{a priori} an $\infty$-category but we will prove that it is in fact a $1$-category. We end the section by defining $1$-motives with coefficients in a flat $\Z$-algebra.
 
\begin{definition}\label{defdel1mot}
Let $S$ be a scheme. 
\begin{enumerate}
    \item An \emph{effective Deligne $1$-motive} over $S$ is a two-term complex $M=[\rho_\sharp(L)\to G]$ in $\rmc((\Sm_S)_{\et},\Z)$ placed in degrees $[0,1]$ with $L$ belonging to $\Loc_S(\Z)$ and $G$ a semi-abelian group scheme (meaning an extension of an abelian scheme by a torus) seen as an étale sheaf on $\Sm_S$. We denote by $\widetilde{\rmm}_1^\D(S)$ the category of effective Deligne $1$-motives.
    \item We will denote by $W_\mathrm{qiso}$ the class of quasi-isomorphisms $[\rho_\sharp(L)\to G]\to [\rho_\sharp(L')\to G']$ of complexes.
    \item The category $\rmm_1^\D(S)$ of \emph{Deligne $1$-motives with torsion} over $S$ is the localization $\widetilde{\rmm}_1^\D(S)$ at $W_\mathrm{qiso}$ (it exists thanks to \cite[Proposition 7.1.3.]{cisinski}).
\end{enumerate}
\end{definition}
A few observations are in order.
\begin{rem}\label{diagramchase}(Compare with \cite[Remarks~C.2.2]{bvk} and \cite[Proposition~1.1.8]{jossen}). By diagram chasing, a map of complexes \[u\colon [\rho_\sharp(L)\to G]\xrightarrow{(u_L,u_G)}[\rho_\sharp(L')\to G']\] is a quasi-isomorphism if and only if the natural maps 
    \[\ker(u_L)\to\ker(u_G) \hspace{1cm} \coker(u_L)\to\coker(u_G)\] are isomorphisms. As $\coker(u_G)$ is finite étale, hence representable by a group scheme which is also connected, it vanishes. 
Hence, we get a commutative diagram with exact rows:
\[\begin{tikzcd}
	0 & E & {\rho_\sharp(L)} & {\rho_\sharp(L')} & 0 \\
	0 & E & G & {G'} & 0
	\arrow[from=1-1, to=1-2]
	\arrow[from=1-2, to=1-3]
	\arrow[equals, from=1-2, to=2-2]
	\arrow[from=1-3, to=1-4]
	\arrow["{u_L}"', from=1-3, to=2-3]
	\arrow[from=1-4, to=1-5]
	\arrow["{u_G}"', from=1-4, to=2-4]
	\arrow[from=2-1, to=2-2]
	\arrow[from=2-2, to=2-3]
	\arrow[from=2-3, to=2-4]
	\arrow[from=2-4, to=2-5]
\end{tikzcd}\]
with $E$ representable by a finite étale group scheme.
\end{rem}
\begin{rem} As the change of sites induces a fully faithful functor $\rmc((\Sm_S)_{\et},\Z)\to \rmc((\Sch_S)_{\mathrm{fppf}},\Z)$, effective Deligne $1$-motives embed fully faithfuly in $\rmc((\Sch_S)_{\mathrm{fppf}},\Z)$. The argument of \Cref{diagramchase} then shows that a map of complexes \[u\colon [\rho_\sharp(L)\to G]\xrightarrow{(u_L,u_G)}[\rho_\sharp(L')\to G']\] becomes a quasi-isomorphism in $\rmc((\Sch_S)_{\mathrm{fppf}},\Z)$ if and only if we have a finite étale group scheme $E$ and a commutative diagram with exact rows as above but in the category $\Sh((\Sch_S)_{\mathrm{fppf}},\Z)$. Since $E$ is étale, the exact sequence of group schemes 
\[0\to E\to G \to G'\to 0\]
stays exact in $\Sh((\Sm_S)_{\et},\Z)$ (see the proof of \cite[Lemma~1.5.2]{bvk}). Hence, the map $u$ already was a quasi-isomorphism. As the converse is true, this shows that $\rmm_1^\D(S)$ can be obtained by seeing $\widetilde{\rmm}_1^\D(S)$ as a subcategory of $\rmc((\Sch_S)_{\mathrm{fppf}},\Z)$ and inverting quasi-isomorphisms there.
\end{rem}
\begin{rem}
The category we construct is possibly different from the essential image of $\widetilde{\rmm}_1^\D(S)$ in the derived category $\D((\Sm_S)_{\et},\Z)$; we do not know whether the obvious functor between those two constructions is an equivalence. 
\end{rem}
A priori $\rmm_1^\D(S)$ is an $\infty$-category but we will now prove that it is in fact a $1$-category. We will prove  that it is in fact a Gabriel-Zisman localization.
\begin{theorem}\label{ComputationHomM1D}
    The $\infty$-category $\rmm_1^\D(S)$ is (equivalent to the nerve of) a $1$-category. It is furthermore an additive category and for any $M$, $N$,
    \[\Hom_{\rmm_1^\D(S)}(M,N)=\colim_{M'\xrightarrow{\mathrm{qiso}} M} \Hom_{\widetilde{\rmm}_1^\D(S)}(M',N).\]
    Furthermore, this colimit is filtered.
\end{theorem}
\begin{proof}
    As isomorphisms are quasi-isomorphisms, \Cref{composition,simplifiable} ensure that $W_\mathrm{qiso}$ satisfies the axioms of Gabriel-Zisman (see \Cref{GZ} below). As quasi-isomorphisms are also closed under composition, \Cref{GZ} ensures that $\rmm_1^\D(S)$ is a $1$-category and that the $\Hom$-sets are of the desired form. They are therefore abelian groups. Furthermore, the computation of the $\Hom$-sets also ensures that finite sums exist and coincide with finite products. 
\end{proof}
The following lemma is essentially contained in \cite[Section~7.2]{cisinski}.
\begin{lemma}\label{GZ}
    Let $\mc{C}$ be a small $1$-category and $W$ be a class of morphisms of $\mc{C}$ that satisfies the axioms of Gabriel-Zisman:
    \begin{enumerate}
        \item Isomorphisms lie in $W$.
        \item Starting from a diagram in $\ccal$ with solid arrows as below
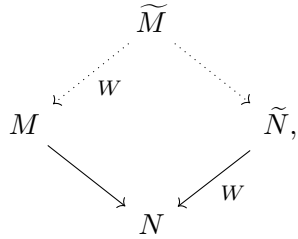
\[\begin{tikzcd}
	& {\widetilde{M}} \\
	M && {\widetilde{N},} \\
	& N
	\arrow["W", dotted, from=1-2, to=2-1]
	\arrow[dotted, from=1-2, to=2-3]
	\arrow[from=2-1, to=3-2]
	\arrow["W", from=2-3, to=3-2]
\end{tikzcd}\]
there is an object $\widetilde{M}$ and dotted arrow filling the diagram and making it commutative.
\item For any pair of morphisms $f,g\colon M\to N$ in $\ccal$, if $t\colon N\to N'$ is a
map in $W$ such that $tf=tg$, there exists a map $s\colon M'\to M$ in $W$ such that $fs=gs$.
    \end{enumerate}
Assume furthermore that $W$ is closed under composition. Then, $\ccal[W^{-1}]$ is a $1$-category and coincides with the classical Gabriel-Zisman localization of $\ccal$. Furthermore, for any $M$, $N$, \[\Hom_{\ccal[W^{-1}]}(M,N)=\colim_{M'\xrightarrow{\mathrm{qiso}} M} \Hom_{\ccal}(M',N).\]
\end{lemma}
\begin{proof}
    Fix an object $M$ of $\ccal$. First note that $W_{\mathrm{qiso}/M}\to \ccal$ is a \emph{putative right calculus of fractions} at $M$ in the sense of \cite[Definition~7.2.2]{cisinski}. We claim that the three axioms above imply that it is a \emph{right calculus of fractions} at $M$ in the sense of \cite[Definition~7.2.6]{cisinski}. Using \cite[Section~7.2.7]{cisinski}, this amounts to showing that the functor 
    \[N\mapsto \colim_{M'\xrightarrow{W}M}\Map_\ccal(M',N)\] sends arrows in $W$ to equivalences. Axiom (2) and the fact that $W$ is closed under composition ensure that this colimit is filtered so that, as the mapping spaces of interest are in fact $\Hom$-sets, this colimit is still a set. Let $N'\to N$ be a map in $W$. Axiom (3) ensures that 
    \[\colim_{M'\xrightarrow{W}M}\Hom_\ccal(M',N')\to \colim_{M'\xrightarrow{W}M}\Hom_\ccal(M',N)\]
    is injective while axiom (2) ensures surjectivity. 

Hence \cite[Theorem~7.2.8]{cisinski} implies that \[\Map_{\ccal[W^{-1}]}(M,N)=\colim_{M'\xrightarrow{\mathrm{qiso}} M} \Hom_{\ccal}(M',N)\] so that $\ccal[W^{-1}]$ is indeed a $1$-category with its hom sets computed in the desired way. It is the classical Gabriel-Zisman localization of $\ccal$ by \cite[Corollary~7.2.12]{cisinski}.
\end{proof}
\begin{lemma}\label{composition}
Let $S$ be a scheme. Then, the category $\rmm_1^\D(S)$ and the class $W_\mathrm{qiso}$ satisfy the second axiom of \Cref{GZ}.
\end{lemma}
\begin{proof}
We can assume that $S$ is connected. 
    Write $N=[\rho_\sharp(L)\to G]$ and $\widetilde{N}=[\rho_\sharp(\widetilde{L})\to \widetilde{G}]$. \Cref{diagramchase} above shows that $\widetilde{L}=(L\times_G \widetilde{G})$. Hence, it suffices to show that if we have a diagram with solid arrow as below with $\Gamma$ semi-abelian
\[\begin{tikzcd}
	& {\widetilde{\Gamma}} \\
	\Gamma && {\widetilde{G}} \\
	& G
	\arrow["\psi", swap, dotted, from=1-2, to=2-1]
	\arrow[dotted, from=1-2, to=2-3]
	\arrow[from=2-1, to=3-2]
	\arrow["\varphi",from=2-3, to=3-2]
\end{tikzcd}\]
there is a semi-abelian group scheme $\widetilde{\Gamma}$ and dotted arrows filling the diagram such that the kernel of $\psi$ a finite étale group scheme. We claim that $\widetilde{\Gamma}:=\Gamma\times_G \widetilde{G}$ is semi-abelian which would finish the proof.

First write 
\[\begin{tikzcd}
    0\ar[r] & \widetilde{T}\ar[r] \ar["\varphi_T",d]& \widetilde{G}\ar[r] \ar["\varphi",d] & \widetilde{A}\ar[r] \ar["\varphi_A",d]& 0 \\
    0\ar[r] & T\ar[r] & G\ar[r] & A \ar[r]& 0
\end{tikzcd}\]
with $T$ and $\widetilde{T}$ tori and $A$ and $\widetilde{A}$ abelian schemes. 
Since $\coker(\varphi_T)$ is connected and finite by the snake lemma, it vanishes. 

By \cite[Exposé~X, Théorème~7.1]{sga3}, as $S$ is connected and noetherian, given a geometric point $\xi$ of $S$, the functor 
\[H\mapsto X(H)_\xi=\Hom_\xi(H_\xi,\mathbb{G}_{m,\xi})\] is an anti-equivalence of categories between groups of multiplicative type over $S$ and $\pi_1^{\mathrm{SGA}3}(S,\xi)$-modules. Hence $\ker(\varphi_T)$ is of multiplicative type hence flat over $S$. As a closed subscheme of $\ker(\varphi)$, it is also finite and unramified over $S$. Hence it is finite étale.

Write now 
\[0\to T'\to \Gamma \to A'\to 0\]
with $T'$ a torus and $A'$ an abelian scheme. Let $T''=T'\times_T \widetilde{T}$, it is obviously of multiplicative type with a discrete Cartier dual $X(H)$ so that its connected component $(T'')^0$ is a torus and the natural map $(T'')^0\to T'$ has finite étale kernel. This yields a commutative diagram with exact rows
\[\begin{tikzcd}
    0\ar[r] & (T'')^0\ar[r] \ar["\varphi_{T'}",d]& \widetilde{\Gamma}\ar[r] \ar["\varphi_\Gamma",d] & \widetilde{\Gamma}/(T'')^0\ar[r] \ar["\varphi_{A'}",d]& 0 \\
    0\ar[r] & T'\ar[r] & \Gamma\ar[r] & A' \ar[r]& 0
\end{tikzcd}\]
where $\ker(\varphi_\Gamma)$ and $\ker(\varphi_{T'})$ are both finite étale, so that by the snake lemma, so is $\ker(\varphi_{A'})$, noting that $\coker(\varphi_{T'})$ vanishes as before. Hence, the algebraic space $\widetilde{\Gamma}/(T'')^0$ is finite étale over $A'$ and therefore representable by a scheme $A''$ which is connected, as $\widetilde{\Gamma}$ is, and proper smooth over $S$ because $A'$ is, and is therefore an abelian scheme. Hence $\widetilde{\Gamma}$ is semi-abelian as needed.
\end{proof}
\begin{lemma}\label{simplifiable}
    The maps in $W_\mathrm{qiso}$ are simplifiable on the left and on the right.
\end{lemma}
\begin{proof}
    The proof is exactly the same as \cite[Proposition~C.2.3]{bvk}.
\end{proof}



We can also recover the usual category of Deligne $1$-motives.
\begin{proposition}\label{usualDel}
Let $M=[\rho_\sharp(L)\to G]$ and $N=[\rho_\sharp(L')\to G']$ be in $\widetilde{\rmm}_1^\D(S)$ and assume that $L'$ is torsion-free. Then the map 
\[\Hom_{\widetilde{\rmm}_1^\D(S)}(M,N)\to \Hom_{\rmm_1^\D(S)}(M,N)\]
is an equivalence.
\end{proposition}
\begin{proof}
     Any diagram with solid arrows as below
\[\begin{tikzcd}
    \widetilde{M} \ar[r,"f"] \ar["\mathrm{qiso}",d, swap]& N \\
    M \ar[ru,dotted]
\end{tikzcd}\]
can be completed with a dotted arrow that makes the diagram commutative. Indeed by \Cref{diagramchase}, letting $\widetilde{M}=[\rho_\sharp(\widetilde{L})\to \widetilde{G}]$, the kernels of the maps $\widetilde{L}\to L$ and $\widetilde{G}\to G$ coincide and are therefore sent to $0$ by $f$.
\end{proof}

\begin{rem}\Cref{usualDel} shows that we recover the usual category of Deligne $1$-motives as a full subcategory of $\rmm_1^\D(S)$.
\end{rem}
\begin{rem}
    Our category coincides with the one that was considered in \cite{BVRS} in the case where $S$ is the spectrum of a field $k$ of characteristic $0$. In particular, we see that $\rmm_1^\D(k)$ is an abelian category by \cite[Proposition 1.3]{BVRS}. 
\end{rem}

In general, we will have to invert the residue characteristic exponents of $S$ for this to stay true. In \cite[Theorem~C.5.3]{bvk}, they indeed show that over a field inverting $p$ yields an abelian category.

\begin{definition}
Let $S$ be a scheme and let $\Lambda$ be a flat $\Z$-algebra.
The category $\rmm_1^\D(S,\Lambda)$ of \emph{Deligne $1$-motives with coefficients $\Lambda$} over $S$ is the localization of the category $\widetilde{\rmm}_1^\D(S,\Lambda)$ of complexes of the form $[\rho_\sharp(L)\to M]\otimes_\Z \Lambda$ at the smallest class $W_\Lambda$ closed under composition and containing $(W_\mathrm{qiso}\otimes_\Z \Lambda)$ and the class $ W_{\mathrm{iso},\Lambda}$ of isomorphisms between objects of $\widetilde{\rmm}_1^\D(S,\Lambda)$.
\end{definition}
\begin{proposition}\label{coeffLambda}
    Deligne $1$-motives with torsion with coefficients $\Lambda$ form a $1$-category which is the Gabriel-Zisman localization of $\widetilde{\rmm}_1^\D(S,\Lambda)$ at $W_\Lambda$. 
    Furthermore, if $M$, $N$ belong to $\widetilde{\rmm}_1^\D(S)$, the canonical map 
\[\Hom_{\rmm_1^\D(S)}(M,N)\otimes_\Z \Lambda \to \Hom_{\rmm_1^\D(S,\Lambda)}(M\otimes_\Z\Lambda,N\otimes_\Z \Lambda)\]
is an equivalence.
\end{proposition}
\begin{proof}
    The class $W_\Lambda$ satisfies the axioms of Gabriel-Zisman localizations (see \Cref{GZ}) by \Cref{composition,simplifiable} and is closed under composition which implies by \Cref{GZ} that $\widetilde{\rmm}_1^\D(S,\Lambda)$ is indeed a $1$-category and that the $\Hom$-sets are computed as stated. 
\end{proof}

\begin{lemma}\label{eliminatingptorsion}
    Let $S$ be a scheme and let $\Lambda$ be a flat $\Z$-algebra. Let $M=[\rho_\sharp(L)\to G]\otimes_\Z \Lambda$ be in $\widetilde{\rmm}_1^\D(S,\Lambda)$. Assume that $p$ is a prime number invertible in $\Lambda$ and let ${}_{p^\infty}L$ be its $p$-primary torsion part and $L^{(p)}=L/{}_{p^\infty}L$ be its prime to $p$ part. Then, there is a map $\rho_\sharp(L^{(p)})\to G$ and a morphism of effective Deligne $1$-motives $[\rho_\sharp(L)\to G]\to [\rho_\sharp(L^{(p)})\to G]$ that becomes an isomorphism after tensoring with $\Lambda$. 
    
    In particular, we can get the category of Deligne $1$-motives with coefficients $\Lambda$ by starting from complexes of the form $[\rho_\sharp(L)\to G]\otimes_\Z \Lambda$ with $L$ that has no $p$-primary torsion.
\end{lemma}
\begin{proof}
    There is an integer $n\geqslant 0$ such that the map ${}_{p^\infty}L\to G\xrightarrow{\times p^n}G$ vanishes. This yields a commutative diagram 
\[\begin{tikzcd}
    L \ar[r]\ar[d] & L^{(p)} \ar[dotted,d] \\
    G \ar["\times p^n",r] & G
\end{tikzcd}\]
whose rows become isomorphisms after tensoring with $\Lambda$. The result follows.
\end{proof}

\begin{rem}
    When $\Lambda=\Q$, \Cref{eliminatingptorsion} shows that the category $\rmm_1^\D(S,\Q)$ can be obtained by starting from the usual category of (torsion-free) Deligne $1$-motives and tensoring the $\Hom$-groups with $\Q$. Hence, we recover the definition of rational Deligne $1$-motives that is used in \cite{plh,plh2,haas}.
\end{rem}

\section{Good reduction for Deligne \texorpdfstring{$1$}{1}-motives with torsion}

Let us denote by $\Z'$ the ring $\Z[1/p\mid p \text{ residue characteristic exponent of }S].$
One of the goals of this section will be to show that Deligne $1$-motives with coefficients $\Z'$ form an abelian category whenever \Cref{etendable} below holds. We must first recall the some constructions on Deligne $1$-motives, namely their $\ell$-adic Tate modules. We will then show that Deligne $1$-motives with coefficients $\Z'$ over a connected normal scheme embed into their counterpart over the generic point and give a condition for a Deligne $1$-motive over the generic point to be of good reduction (that is to come from an object on the whole space). These results extend \cite[Proposition~A.11]{plh} and \cite[Theorem~4.10]{haas} to the case of $\Z'$-coefficients.

Note now that the natural functor $\widetilde{\rmm}_1^\D(S,\Lambda)\to \D((\Sm_S)_{\et},\Lambda)$ sends the maps in $W_\Lambda$ to equivalences and therefore factors uniquely through a map 
\[L_\D\colon \rmm_1^\D(S,\Lambda)\to \D((\Sm_S)_{\et},\Lambda).\]

\begin{definition}
    Let $S$ be a scheme and let $\ell$ be a prime number invertible on $S$. 
    The \emph{Tate module functor} is defined as the composition 
    \[T_\ell\colon\rmm_1^\D(S)\xrightarrow{L_\D} \D((\Sm_S)_{\et},\Z)
    \xrightarrow{(-)^\wedge_\ell}\D((\Sm_S)_{\et},\Z)^\wedge_\ell
    \xrightarrow{\rho^\sharp}\D(S_{\et},\Z)^\wedge_\ell\xrightarrow{\nu^*}\D(S_\proet,\Z_\ell)\]

\end{definition}
\begin{proposition}\label{Tatemodule}
    Let $S$ be a scheme, and let $\ell$ be a prime number invertible on $S$. The Tate module functor lands in $\Loc_S(\Z_\ell)$. 
    
    If $\Lambda$ is a flat $\Z$-algebra, it therefore gives rise to a functor \[T_\ell\colon \rmm_1^\D(S,\Lambda)\to \Loc_S(\Lambda\otimes_\Z \Z_\ell)\]
that we also call the Tate module functor.
\end{proposition}
\begin{proof}We can assume $\Lambda=\Z$. 
Take now a Deligne $1$-motive with torsion $M=[\rho_\sharp(L)\to G]$. Then, we have an exact triangle 
\[G[-1]\to L_\D(M)\to \rho_\sharp(L)\]
so we can assume $M$ to be either of the form $[L\to 0]$ or $[0\to G]$. The first case is \cite[Proposition 3.3.9]{AM1} and the second case is \cite[Proposition 5.1 (2)]{ahplh}.
\end{proof}

\begin{hyp}\label{etendable}
    Let $S$ be a connected normal scheme with generic point $\eta$, and let $\ell$ be a prime number invertible at $\eta$. Then any abelian variety over $\eta$ whose rational $\ell$-adic Tate module $T_\ell(A)\otimes_\Z \Q$ extends to a local system on $S[1/\ell]$ (\textit{i.e.} is unramified) extends to an abelian scheme on $S[1/\ell]$ (\textit{i.e.} has good reduction).
\end{hyp}
\begin{rem}\label{whendoesetendablehold}
    The rational $\ell$-adic Tate module $T_\ell(A)\otimes_\Z \Q$ extends to a local system on $S[1/\ell]$ if and only if $T_\ell(A)$ also extends as the latter is a lattice inside of $T_\ell(A)\otimes_\Z \Q$ and the action of the Galois group $\Gal(\overline{\eta}/\eta)$ at $\eta$ factors through the étale fundamental group $\pi_1^{\et}(S,\overline{\eta})$ if and only if this is true over some Galois-invariant lattice. In particular, \Cref{etendable} is know to hold when $S$ is a Dedekind scheme by the criterion of Néron, Ogg and Shafarevich \cite{Serre-Tate} and when $S$ is a $\Q$-scheme by \cite[Corollaire~4.2]{Grothendiecketendable}. It is also known to be false over $\mathbb{A}^2_k$ for $k$ a field of characteristic $p$ (see \cite[Remarques~4.6]{Grothendiecketendable}).
\end{rem}

If $f\colon T\to S$ is a morphism of scheme and $\Lambda$ is a flat $\Z$-algebra, then the pullback of group schemes induces a map 
\[f^*\colon \rmm_1^\D(S,\Lambda)\to \rmm_1^\D(T,\Lambda)\]
which maps $[\rho_\sharp(L)\to G]$ to $[\rho_\sharp(f^*(L))\to f^*(G)]$.
\begin{proposition}\label{f^*}
    Let $S$ be a scheme and let $\Lambda$ be a flat $\Z$-algebra. The diagram 
    \[
    \begin{tikzcd}
        \rmm_1^\D(S,\Lambda) \ar["f^*",r]\ar["L_\D",d] & \rmm_1^\D(T,\Lambda) \ar["L_\D",d]\\
        \D((\Sm_S)_{\et},\Lambda) \ar["f^*",r]& \D((\Sm_T)_{\et},\Lambda)
    \end{tikzcd}
    \]
    is commutative.
\end{proposition}
\begin{proof}This can be derived from the case of rational coefficients and torsion coefficients as in the proof of \cite[Theorem~5.11]{rosassoto}. 
Alternatively, S. Tubach pointed out to me that it can be seen directly using Breen-Deligne resolutions (see \cite[Theorem~4.5]{condensed}).
\end{proof}

The following proposition is the first step in generalizing \cite[Proposition A.11]{plh} to $\Z'$-coefficients.
\begin{proposition}\label{embeddinggeneric}
    Let $S$ be a connected normal scheme with generic point $\eta$ and let $\Lambda$ be a flat $\Z$-algebra. The pullback map 
    \[\rmm_1^\D(S,\Lambda)\to \rmm_1^\D(\eta,\Lambda)\]
    is faithful.
\end{proposition}
\begin{proof} We can assume $\Lambda=\Z$. 
Letting $M=[\rho_\sharp(L)\to G]$ and $N=[\rho_\sharp(L')\to G']$ be objects of $\widetilde{\rmm}_1^\D(S)$, we want to prove that the natural map
 \[\Hom_{\rmm_1^\D(S)}(M,N)\to \Hom_{\rmm_1^\D(\eta)}(M_\eta,N_\eta)\]
is injective. 
 As we have a functorial exact sequence:
 \[0\to\Hom_{\widetilde{\rmm}_1^\D(S)}(-,G'[-1])\to \Hom_{\widetilde{\rmm}_1^\D(S)}(-,N)\to \Hom_{\widetilde{\rmm}_1^\D(S)}(-,\rho_\sharp(L')[0])\]
 we get an exact sequence (note that the colimit that defines our $\Hom$-sets is filtered by \Cref{composition}) 
 \[0\to\Hom_{\rmm_1^\D(S)}(M,G'[-1])\to \Hom_{\rmm_1^\D(S)}(M,N)\to \Hom_{\rmm_1^\D(S)}(M,\rho_\sharp(L')[0]).\]
The same remains true over $\eta$. But the maps in the following diagram
 \[\begin{tikzcd}
     \Hom_{\widetilde{\rmm}_1^\D(S)}(M,G'[-1]) \ar[r]\ar[d]& \Hom_{\widetilde{\rmm}_1^\D(S)}(G[-1],G'[-1])\ar[d]\\
     \Hom_{\rmm_1^\D(S)}(M,G'[-1]) \ar[r]& \Hom_{\rmm_1^\D(S)}(G[-1],G'[-1])
 \end{tikzcd}\]
are equivalences: the top horizontal arrow is an equivalence because everything is happening in the category of chain complexes while the vertical arrows are equivalences by \Cref{usualDel}. The same fact is also true over $\eta$ and by \cite[Lemme 7.2.2.1]{ega1}, the map 
\[\Hom_{\Sh((\Sm_S)_{\et},\Z)}(G,G')\to \Hom_{\Sh((\Sm_\eta)_{\et},\Z)}(G_\eta,G'_\eta)\]
is an equivalence because $S$ is normal. Hence, we are reduced to the case $N=\rho_\sharp(L')[0].$ Now, by \Cref{diagramchase} and \Cref{ComputationHomM1D}
 \begin{align*}
     \Hom_{\rmm_1^\D(S)}(M,N)
     &=\colim_{\widetilde{G}\xrightarrow{\varphi} G} \ \Hom_{\widetilde{\rmm}_1^\D(S)}([\rho_\sharp(L\times_G \widetilde{G})\to \widetilde{G}],\rho_\sharp(L')[0]) \\
     &=\colim_{\widetilde{G}\xrightarrow{\varphi} G}\ \Hom_{\Loc_S(\Z)}(L\times_G \widetilde{G},L')
 \end{align*}
 where the colimit runs through those surjective maps of semi-abelian schemes $\widetilde{G}\xrightarrow{\varphi} G$ whose kernel $\ker(\varphi)$ is finite étale. From the exact sequence \[0\to \ker(\varphi)\to (L\times_G \widetilde{G}) \to L\to 0,\] we get an exact sequence 
 \[0 \to \Hom_{\Loc_S(\Z)}(L,L') \to \Hom_{\rmm_1^\D(S)}(M,N) \to \colim_{\widetilde{G}\xrightarrow{\varphi} G}\ \Hom_{\Sh((\Sm_S)_{\et},\Z)}(\ker(\varphi),\rho_\sharp(L')) \]
 and we are reduced to the case where $M=G[-1]$. The faithfulness in the cases of étale local systems follows from \cite[Lemme 7.2.2.1]{ega1}. Hence, we are reduced to showing that if $L$ is a torsion étale local system on $S$ and if $G$ is a semi-abelian $S$-group scheme, the map 
\[\colim_{\widetilde{G}\xrightarrow{\varphi} G}\Hom(\ker(\varphi),\rho_\sharp(L'))\to \colim_{G'\xrightarrow{\psi} G_\eta}\Hom(\ker(\psi),\rho_\sharp(L'))\]
where the colimit on the left runs through those finite étale $\widetilde{G}\xrightarrow{\varphi} G$ and the colimit on the right runs through those finite étale $G'\xrightarrow{\varphi} G_\eta$. We now claim that the maps from each of these $\Hom$-groups to the colimit is injective. As $\ker(\varphi_\eta)=\ker(\varphi)_\eta$, this claim implies the desired injectivity. Let us now prove the claim, namely that for any $\varphi_0\colon G_0\to G$ finite étale, the map
\[\Hom(\ker(\varphi_0),\rho_\sharp(L'))\to\colim_{\widetilde{G}\xrightarrow{\varphi} G}\Hom(\ker(\varphi),\rho_\sharp(L'))\]
is injective. If $G_1\to G_0\to G$ is such that $\varphi_1\colon G_1\to G$ is finite étale and the map $\ker(\varphi_1)\to \ker(\varphi_0)\to \rho_\sharp(L')$ vanishes, then the map $\ker(\varphi_0)\to \rho_\sharp(L')$ also vanishes: $G_1\to G_0$ is finite étale and therefore surjective. 
\end{proof}
We now want to prove that the functor of \Cref{embeddinggeneric} is fully faithful with $\Z'$-coefficients. We will need the following result.
\begin{proposition}\label{HomcoeffZ'}Let $S$ be a scheme.
If $M=[\rho_\sharp(L)\to G]$ and $N$ belong to $\widetilde{\rmm}_1^\D(S,\Z)$, we have a canonical isomorphism
\[\Hom_{\rmm_1^\D(S,\Z')}(M\otimes_\Z \Z',N\otimes_\Z\Z')=\colim_{G\xrightarrow{\times n}G} \Hom_{\widetilde{\rmm}_1^\D(S,\Z)}([\rho_\sharp(L)\times_G G\to G],N)\otimes_\Z \Z'\]
where $n$ runs through the set of integers which are invertible on $S$ (or equivalently non-invertible in $\Z'\setminus\{0\}$). 
\end{proposition}
\begin{proof}
    By \Cref{ComputationHomM1D} and \Cref{coeffLambda}, we have a canonical isomorphism
\[\Hom_{\rmm_1^\D(S,\Z')}(M\otimes_\Z \Z',N\otimes_\Z\Z') =\colim_{\widetilde{G}\xrightarrow{\varphi}G} \Hom_{\widetilde{\rmm}_1^\D(S,\Z)}([\rho_\sharp(L)\times_G \widetilde{G}\to \widetilde{G}],N)\otimes_\Z \Z'\]
with $\varphi\colon \widetilde{G}\to G$ finite étale and $\widetilde{G}$ semi-abelian. Fix such a $\varphi\colon \widetilde{G}\to G$. 
Write 
\[\begin{tikzcd}
    0\ar[r] & \widetilde{T}\ar[r] \ar["\varphi_T",d]& \widetilde{G}\ar[r] \ar["\varphi",d] & \widetilde{A}\ar[r] \ar["\varphi_A",d]& 0 \\
    0\ar[r] & T\ar[r] & G\ar[r] & A \ar[r]& 0
\end{tikzcd}\]
with $T$ and $\widetilde{T}$ tori and $A$ and $\widetilde{A}$ abelian schemes. In the proof of \Cref{composition}, we saw that $\ker(\varphi_T)$ is finite étale. By \cite[Corollary~B.3.3]{conrad}, it is therefore of multiplicative type, and therefore up to passing to a finite étale cover, it is a product of $\mu_n$s where the $n$s need to be invertible on $S$ by étaleness. The group-scheme $\ker(\varphi_A)$ is also finite étale as a quotient of finite étale group schemes. Write \[\ker(\varphi)=K\oplus K'\] with $K$ of $n$-torsion, $n$ invertible on $S$ and $K'$ made of $n'$-torsion with $n'$ a product of residue characteristic exponents of $S$. The map of algebraic spaces $\overline{\varphi}\colon \widetilde{G}/K'\to G$ is finite étale and thus $\widetilde{G}/K'$ is a commutative group-scheme. We claim that it is semi-abelian. Indeed, it fits in an exact sequence 
\[0\to \widetilde{T}\to \widetilde{G}/K' \to \widetilde{A}/K'\to 0\]
and the map $\widetilde{A}/K'\to A$ is finite étale as its kernel is $\ker(\varphi_A)/K'$ which is finite étale, hence $\widetilde{A}/K'$ is an abelian scheme. 

Now note that the map \[[\rho_\sharp(L)\times_G\widetilde{G}\to \widetilde{G}]\otimes_\Z \Z'\to \rho_\sharp(L)\times_G(\widetilde{G}/K')\to \widetilde{G}/K']\otimes_\Z \Z\] is an isomorphism as $K'\otimes_\Z \Z'=0$. Hence, in the colimit, we may restrict to those maps $\varphi\colon \widetilde{G}\to G$ whose kernel is of $n$-torsion with $n$ invertible on $S$. Now, for any such map, we have a commutative diagram 
\[\begin{tikzcd}
    G\ar["\times n",r]\ar[dotted,d] & G \\
    \widetilde{G}\ar["\varphi",ur]
\end{tikzcd}\]
The result then follows from the fact that the multiplication by $n$ map on $G$ is finite étale: it is étale by \cite[Section~7.3, Lemma~2]{neronmodels} and it therefore suffices to show that it is proper in both the case of tori and the case of abelian schemes. As being proper is fpqc local on the base by \cite[Tag~02L1]{stacks}, the case of tori reduces to that of $\mathbb{G}_{m,S}$ where it is true because $\mu_{n,S}$ is indeed finite. The case of abelian schemes is also straightforward because they are proper on $S$ so any map between them is proper.
\end{proof}
We can now prove fullness.
\begin{theorem}\label{embeddinggeneric2}
    Let $S$ be a connected normal scheme with generic point $\eta$ and let $\Lambda$ be a flat $\Z'$-algebra. 
Then, the restriction to the generic point 
\[\rmm_1^\D(S,\Lambda)\to \rmm_1^\D(\eta,\Lambda)\]
is fully faithful.
\end{theorem}
\begin{proof}
Let us start by proving that if we have objects $M=[\rho_\sharp(L)\to G]$ and $ M'=[\rho_\sharp(L')\to G']$
and an effective map $M_\eta\to M'_\eta$, we have an effective map $M\to M'$ that restricts to $M_\eta\to M'_\eta$. Our map indeed yields maps $L_\eta \to L'_\eta$ and $G_\eta \to G'_\eta$ which extend uniquely to maps $L\to L'$ and $G\to G'$ by \cite[Exposé~V, Proposition~8.2]{sga1} and \cite[Proposition A.11]{plh} respectively. We have to prove that the resulting maps $\rho_\sharp(L)\to G'$ coincide knowing that they coincide at $\eta$. It suffices to show that a map $f\colon \rho_\sharp(L)\to G'$ that vanishes at $\eta$ vanishes. This reduces to proving that the induced map to the abelian part of $G'$ vanishes and then proving that the resulting map to the toric part of $G'$ vanishes. 

In the toric case we can replace $S$ with a finite étale cover and therefore assume that $G'$ is split and $L$ is trivial. We may then further assume that $L$ is of the form $\Z/n\Z_S$ (with $n$ possibly $0$) and $G'=\mathbb{G}_{m,S}$. Hence the data of a map $\rho_\sharp(L)\to G'$ is the same as that of an element $x$ of $\mathcal{O}_S^\times(S)$ such that $x^n=1$, as the map vanishes at $\eta$, this element $x$ is the unit in the residue field $k(\eta)$ whence in $\mathcal{O}_S^\times(S)$ by integrality.

Hence, we can assume $G'$ to be an abelian scheme. By \cite[Théorème~8.8.2(i)]{ega4}, there is an open subset $U$ of $S$ such that $f|_U$ vanishes. The closed complement $Z$ of $U$ has finitely many points which are of codimension $1$ in $S$, denote them by $s_1,\ldots,s_r$ and let $S_i=\Spec(\mathcal{O}_{S,s_i})$. Since $S$ is normal, $S_i$ is the spectrum of a discrete valuation ring. Hence, by \cite[Section~1.2, Proposition~8]{neronmodels}, the $G'|_{S_i}$ is a Néron model of $G'_\eta$ and thus $f|_{S_i}$ vanishes. By \cite[Théorème~8.8.2(i)]{ega4} again, there is an open neighborhood $U_i$ of $S_i$ such that $f|_{U_i}$ vanishes. Hence, we can assume that $U$ contains all points of codimension $1$ (hence all points of depth $1$ as $S$ is normal) and apply \cite[Collaire~IX.1.4]{raynaudfaisceauxamples} which yields that $f$ vanishes.

To prove the result, we can assume that $\Lambda=\Z'$.
We know that our functor is faithful by \Cref{embeddinggeneric}. 
 Assume given a map $M_\eta\otimes_\Z\Z'\to M'_\eta\otimes_\Z\Z'$ in $\rmm_1^\D(\eta,\Z')$. Using \Cref{HomcoeffZ'}, there is an integer $n$ invertible on $S$ such that this map comes from maps \[M_\eta \xleftarrow{[n]}M_\eta^{(n)}\to M'_\eta\]
 up to multiplication by $\frac{1}{m}$ for some $m$ invertible in $\Z'$ and with $M_\eta^{(n)} \xrightarrow{[n]}M_\eta$ the map given by pulling back the multiplication by $n$ map $G_\eta \to G_\eta$. Both of these maps extend by the effective case, so that we get a map $M\otimes_\Z \Z'\to M'\otimes_\Z \Z'$ in $\rmm_1^\D(S,\Z')$ extending the previous one. Hence the functor is full which yields the result.
\end{proof}

\begin{definition}
    Let $S$ be a connected normal scheme with generic point $\eta$ and let $\Lambda$ be a flat $\Z$-algebra. We say that a Deligne $1$-motive with coefficients $\Lambda$ over $\eta$ has \emph{good reduction} if it lies in the essential image of the map $\rmm_1^\D(S,\Lambda)\to \rmm_1^\D(\eta,\Lambda)$.
\end{definition}

The following proposition generalizes \cite[Theorem~4.10]{haas} to Deligne $1$-motives with torsion.
\begin{proposition}\label{bonnereduction}
    Let $S$ be a connected normal scheme with generic point $\eta$ of characteristic exponent $p$ and let $\Lambda$ be a flat $\Z'$-algebra. Assume that \Cref{etendable} holds. 
    A Deligne $1$-motive with coefficients $\Lambda$ over $\eta$ has good reduction if and only if for any prime number $\ell\neq p$, its $\ell$-adic Tate module has good reduction on $S[1/\ell]$. 
\end{proposition}
\begin{proof}
The only if part is obvious. 
To prove the converse, we can assume that $\Lambda=\Z'$. Let now \[M_\eta=[\rho_\sharp(L_\eta)\xrightarrow{u_\eta}G_\eta]\otimes_\Z \Z'\] be in $\rmm_1^\D(\eta,\Z')$ and assume that all of its Tate modules extend to locally constant $\Z_\ell\otimes_\Z \Z'$-sheaves on $S[1/\ell]$ for any $\ell\neq p$. We can assume that the torsion part of $L_\eta$ is invertible on $S$ by \Cref{eliminatingptorsion}.

We have an exact sequence
\[0\to T_\ell(G_\eta[-1])\otimes_\Z \Z'\to T_\ell(M_\eta)\to L_\eta\otimes_\Z (\Z'\otimes_\Z\Z_\ell) \to 0\]
In particular, the rational Tate modules of $G_\eta$ extends to $S[1/\ell]$. Since \Cref{etendable} holds, we can use \cite[Theorem~4.10]{haas} to prove that $G_\eta$ extends to a semi-abelian group scheme over $S[1/\ell]$. As pulling back to $\eta$ is fully faithful by \Cref{embeddinggeneric2}, these sheaves coincide on the intersections and therefore $G_\eta$ extends to a semi-abelian group scheme $G$ over $S$. The local system $L_\eta$ also extends as the map $\Gal(\overline{\eta}/\eta)\to \mathrm{Aut}(L_\eta)$ (with $\Gal(\overline{\eta}/\eta)$ the Galois group of the residue field at $\eta$) that defines it factors through the étale fundamental group $\pi_1^{\et}(S,\overline{\eta})$ as this is true after tensoring with $\Z_\ell$ for any $\ell$ that is invertible on $S$ (and thus any non-invertible prime in $\Z'$) and in the case where there is no such prime, the local system $L$ is a lattice and a similar argument works for any $\ell\neq p$. Thus we get a locally constant sheaf $L$ whose fiber at $\eta$ is $L_\eta$.

We therefore have to prove that the map $u_\eta$ extends to a (necessarily unique by \Cref{embeddinggeneric2}) map $u\colon \rho_\sharp(L)\to G$. To see this we once again build up from the case of torsion-free Deligne $1$-motives: consider an exact sequence 
\[L_2\to L_1\to L\to 0\]
with $L_i$ torsion-free. We get two torsion-free Deligne $1$-motives $(M_i)_\eta=[\rho_\sharp((L_i)_\eta)\to G_\eta]$ that extend to $S$ by \cite[Theorem~4.10]{haas}. Note that the map $\rho_\sharp((L_2)_\eta)\to G_\eta$ is the zero map so that the canonical map $(M_2)_\eta \to \rho_\sharp((L_2)_\eta)$ splits. By full faithfulness (\Cref{embeddinggeneric}), the canonical map $M_2\to\rho_\sharp(L_2)$ also splits so that the map $\rho_\sharp(L_2)\to G$ is the zero map. In particular, the map $\rho_\sharp(L_1)\to M$ factors through $\rho_\sharp(L)$ which finishes the proof.
\end{proof}

\begin{corollary}\label{abeliancat}
    Let $S$ be a connected normal scheme with generic point $\eta$ and let $\Lambda$ be a localization of $\Z'$. Assume that \Cref{etendable} holds. Then the category $\rmm_1^\D(S,\Lambda)$ is a Serre subcategory of $\rmm_1^\D(\eta,\Lambda)$ through the embedding of \Cref{embeddinggeneric2}. In particular, it is an abelian category.
\end{corollary}
\begin{proof}
    For any $\ell$ invertible on $S$, the subcategory $\Loc_S(\Z_\ell\otimes_\Z\Lambda)$ of $ \Loc_\eta(\Z_\ell\otimes_\Z\Lambda)$ is Serre. The result then follows from \Cref{bonnereduction}.
\end{proof}

\section{Voevodsky \texorpdfstring{$1$}{1}-motives}
Our next goal is to compare Deligne $1$-motives to Voevodsky motives. First we recall various categories of $1$-motives from \cite{plh} and their stabilities with respect to the six functors. We then prove that the $\ell$-adic realization functors form a  conservative family on $1$-motives. Finally define a functor $\Phi_S$ from Deligne $1$-motives to Voevodsky motives and show that its image lands in the  subcategory of smooth $1$-motives. This generalizes \cite[Proposition~2.15]{plh}.

 \begin{definition}\label{def AM} Let $S$ be a scheme and let $\Lambda$ be a commutative ring. We define
\begin{enumerate}
\item The category $\DM_{\et}^1(S,\Lambda)$ of \emph{$1$-motives} to be the thick subcategory of $\DM_{\et}(S,\Lambda)$ generated by the $f_\sharp(\Lambda_X)$ for $f\colon X\to S$ smooth of relative dimension at most $1$.
\item The category $\DM_{\et}^{\mathrm{ind}1}(S,\Lambda)$) \emph{Ind-$1$-motives} to be the localizing subcategory of $\DM_{\et}(S,\Lambda)$ generated by $\DM_{\et}^1(S,\Lambda)$.
\item The category $\DM_{\et}^{\sm1}(S,\Lambda)$ of \emph{smooth $1$-motives} to be the intersection of $\DM_{\et}^1(S,\Lambda)$ with the category of dualizable objects of $\DM_{\et}(S,\Lambda)$.
\end{enumerate}
\end{definition}
\begin{rem}
    In \cite{plh}, what we call $1$-motives is called "constructible homological $1$-motive" and is denoted by $\mathrm{DA}_1$. There is also a notion of cohomological $1$-motives denoted by $\mathrm{DA}^1$ and defined with the $f_*(\Lambda_X)$ for $f\colon X\to S$ proper as generators. According to \cite[Proposition~1.28]{plh} cohomological $1$-motives are exactly the $M(-1)$ with $M$ a homological $1$-motive so all our result will also apply to cohomological $1$-motives up to a twist by $(-1)$. In \cite{haas}, Haas only considers cohomological $1$-motives but his results will apply to our setting for the same reason.
\end{rem}

The functors $\otimes$ and $f^*$ (where $f$ is any morphism) induce functors over the categories $\DM_{\et}^{(\sm)1}(-,\Lambda)$. 
The following proposition is \cite[Proposition~1.17]{plh}, note that in \textit{loc. cit.} it is only formulated with $\Q$ coefficients but the proof works over any ring of coefficients.
\begin{proposition}\label{f_!} Let $\Lambda$ be a commutative ring. The categories $\DM_{\et}^{(\mathrm{ind})1}(-,\Lambda)$ are closed under the the functors of type $f_!$, with $f$ is a quasi-finite morphism. In particular, the fibered category $\DM_{\et}^{(\mathrm{ind})1}(-,\Lambda)$ satisfies the localization property \eqref{localization}.
\end{proposition}
The $\ell$-adic realization functors form a conservative family of functors on $1$-motives. 
\begin{proposition}\label{conservativityconj}
    Let $S$ be a scheme and let $\Lambda$ be a commutative ring. The family of \emph{reduced $\ell$-adic realization} functors
    \[\overline{\rho}_\ell\colon \DM_{\et}^1(S,\Lambda)\to \DM_{\et}^1(S[1/\ell],\Lambda)\xrightarrow{\rho_\ell}\D_\mathrm{cons}(S,\Lambda\otimes_\Z\Z_\ell)\]
    is conservative when $\ell$ runs through the set of all prime numbers.
\end{proposition}
\begin{proof}
Let $M$ be a $1$-motive such that $\rho_\ell(M)=0$, then $\rho_\ell(M)\otimes_\Z\Q=0$ so that the image of $M\otimes_\Z \Q$ through the rational $\ell$-adic realization  
\[\DM_{\et}^\mathrm{gm}(S,\Lambda\otimes_\Z\Q)\to \D_\mathrm{cons}(S[1/\ell],\Lambda\otimes_\Z\Q_\ell),\]
which we recall to be obtained from the integral one by tensoring with $\Perf_\Q$ over $\Perf_\Z$ (this is the same as tensoring the mapping spectra by $\Q$ and idempotent-completing by \cite[Propo-
sition~3.5.5]{tensoringwithperf}), vanishes by \cite[Theorem~4.1(iv)]{plh2} . Hence $(M\otimes_\Z \Q)|_{S[1/\ell]}$ vanishes. As this is true for any $\ell$, $M\otimes_\Z \Q$ vanishes. On the other hand, as $\rho_\ell(M\otimes_\Z \Z/\ell\Z)=0$, the rigidity theorem \cite[Corollary~3.2]{bachmannrigidity} together with the fact that motives of $\ell$-torsion are supported on $S[1/\ell]$ by \cite[Proposition~A.3.4]{em} imply that $M\otimes_\Z \Z/\ell\Z=0$ for any $\ell$. Hence, $M=0$.
\end{proof}

Assume now that $\Lambda$ is a flat $\Z$-algebra. We have a map 
\[\Phi_S\colon \rmm_1^\D(S,\Lambda)\to \DM_{\et}(S,\Lambda)\]
given by $\Sigma^\infty L_{\mathbb{A}^1}L_\D$. 

\begin{theorem}\label{Delignemotare1mot}
    The essential image of $\Phi_S$ is contained in $\DM_{\et}^{\sm1}(S,\Lambda)$
\end{theorem}
\begin{proof}
We can assume $\Lambda=\Z$. Let $M=[\rho_\sharp(L)\to G]$ be a Deligne $1$-motive. Using \cite[Proposition~3.1.9]{AM1}, $\Sigma^\infty L_{\AAA}\rho_\sharp(L)$ belongs to the thick subcategory generated by the $f_\sharp(\Lambda_X)$ for $f\colon X\to S$ finite étale. 

Write now \[0\to T\to G\to A\to 0\] an exact sequence with $T$ a torus and $A$ an abelian scheme. \Cref{toruscase} below ensures that $\Sigma^\infty L_{\AAA}(T)$ belongs to $\DM_{\et}^{\sm1}(S,\Z)$. 

We now claim that $\Sigma^\infty L_{\AAA}(A)$ is a $1$-motive. We can assume $S$ to be reduced. Furthermore, using the localization property \eqref{localization}, we can replace $S$ with the neighborhood of any generic point $\eta$ and hence assume $S$ to be regular and connected. By \cite[Theorem~11]{spacefillingcurveskatz}, there is a proper, smooth, geometrically connected curve $p\colon C\to \eta$ with a rational point and an isogeny $\mathrm{Jac}(C/\eta)\to A_\eta$ from the jacobian variety of $C$ over $\eta$. Using \cite[Lemma~4.7]{plh}, there is an étale neighborhood $W$ of $\eta$ such that there is a proper, smooth relative curve $p\colon C\to W$ with geometrically connected fibers, a section and a surjective map $\mathrm{Jac}(C/W)\to A_W$ with finite étale kernel where $\mathrm{Jac}(C/W)$ denotes the jacobian abelian scheme of $C$ over $W$. \Cref{curvedecomposition} below then shows that $\mathrm{Jac}(C/W)$ and hence $\Sigma^\infty L_{\A^1}(A_W)$ is a (smooth) $1$-motive. In particular, the motive $\Sigma^\infty L_{\A^1}(A_W)$ is étale-locally a geometric motive\footnote{This is the same as being a geometric motive by \cite[Theorem~4.1]{etalemotgeo} but we will not use that fact.} Now, the motive $\Sigma^\infty L_{\A^1}(A_\eta)$ is built out of finitely many $p_\sharp(\Z_X)$ for $p\colon X\to \eta$ smooth of relative dimension at most $1$. Up to replacing $S$ with an open subset \cite[Lemma~1.24]{plh} ensures that these extend to smooth schemes of relative dimension at most $1$ over $S$. Now by continuity for étale-locally a geometric motive in the form of \cite[Proposition~6.3.7]{em}, we see that up to shrinking $S$ again $\Sigma^\infty L_{\A^1}(A)$ is now built out of these extended schemes. Hence it is a $1$-motive.

Being dualizable means that for any $N$, the map 
\[\underline{\Hom}(\Sigma^\infty L_{\A^1}(A),\Z)\otimes N\to \underline{\Hom}(\Sigma^\infty L_{\A^1}(A),N)\]
is an equivalence. 
As we have proved that $\Sigma^\infty L_{\A^1}(A)$ is a geometric motive,  can be tested after tensoring with $\Q$ and $\Z/\ell\Z$ for any prime number $\ell$ by \cite[Corollary~5.4.11]{em}. It is true rationally because $\Sigma^\infty L_{\A^1}(A\otimes \Q)$ is a direct summand of the rational motive of $A$ by \cite[Theorem~3.3]{ahplh}, and with torsion coefficients, we have $\Sigma^\infty L_{\A^1}(A)\otimes_\Z \Z/\ell\Z=\rho_\sharp({}_\ell A)$ which is a local system hence dualizable.
\end{proof}
\begin{lemma}\label{toruscase}
    Let $T/S$ be a torus, and $\rho_\sharp(X_*(T))$ its cocharacter lattice. There is an isomorphism
    \[\Sigma^\infty L_{\AAA}(T)\simeq \rho_!(X_*(T))(1)[1].\] 
\end{lemma}
\begin{proof}
    The proof of \cite[Corollary~2.13]{plh} allows us to reduce to the case of $\mathbb{G}_m$ which holds over $S=\Spec(\Z)$ by \cite[Proposition~11.2.11]{tcmm} and hence over any scheme $S$ because the canonical map \[\Z_S(1)[1]\to \Sigma^\infty L_{\AAA}(\mb{G}_{m,S}) \] is obtained by pulling back that of $\Spec(\Z)$ (here we use \Cref{f^*}).
\end{proof}

\begin{lemma}\label{curvedecomposition}
    Let $S$ be a regular scheme and $p\colon C\to S$ be a smooth projective curve with geometrically connected fibers and a section $\sigma\colon S\to C$. Then, there is a canonical decomposition
    \[p_\sharp(\Z_C)\simeq \Z_S \oplus \Sigma^\infty L_{\A^1}\mathrm{Jac}(C/S)\oplus \Z_S(1)[2]\] 
    where $\mathrm{Jac}(C/S)$ denotes the jacobian abelian scheme of $C$ over $S$.
\end{lemma}
\begin{proof}Recall first that we have a canonical equivalence
\[p_\sharp(\Z_C)\simeq p_*(\Z_C)(1)[2].\]
The section $\sigma$ gives canonical splittings of the maps \[\Z_S\to p_\sharp(\Z_C) \hspace{1cm}p_*(\Z_C)\to \Z_S.\]
Hence, we get decompositions 
\[\Z_S\oplus K_1\simeq p_\sharp(\Z_C)\simeq p_*(\Z_C)(1)[2]\simeq \Z_S(1)[2]\oplus K_2.\]
As $\Hom(\Z_S(1)[2],\Z_S)=0$ by \cite[Proposition~11.1]{ayo14} (notice that the assumptions on the cohomological dimension it \textit{loc. cit.} can be removed because the rigidity theorem \cite[Corollary~3.2]{bachmannrigidity} now applies to this level of generality), the map $\Z_S(1)[2]\to p_\sharp(\Z_C)$ factors through $K_1$ and we therefore get a splitting $K_1\simeq \Z_S(1)[2]\oplus K.$
It remains to identify $K$. 

We now recall some constructions from \cite[Section~2.3]{plh}. 
By adjunction, we get a map 
\[\Sigma^\infty(\Omega^\infty(p_*\Z_C(1)))(-1)\to p_*\Z_S\]
As $\Sigma^\infty(-)(-1)$ commutes with pullbacks, its right adjoint commutes with pushforwards. Furthermore, the isomorphism $\Sigma^\infty L_{\A^1}(\mb{G}_{m,C})[-1]\xrightarrow{\sim} \Z_C(1)$ yields a map $L_{\A^1}(\mb{G}_{m,C})[-1]\to \Omega^\infty(\Z_C(1))$ by adjunction. Hence, we get a map \[\alpha\colon \Sigma^\infty(p_*(L_{\A^1}(\mb{G}_{m,C})))\to p_*\Z_S(1)[1].\]
We claim that $\alpha$ is an equivalence. Indeed, this can be tested after tensoring with $\Q$ and with $\Z/\ell\Z$ for any prime number $\ell$. The first case is \cite[Theorem~3.15]{plh} combined with \cite[Lemma~2.27]{plh}. After tensoring with $\Z/\ell\Z$, the functor $\Sigma^\infty$ becomes an equivalence by rigidity \cite[Corollary~3.2]{bachmannrigidity} and therefore commutes with $p_*$ and with twists (which amount to tensoring with $\mu_\ell$ in that case). 
Now the square:
\[\begin{tikzcd}
    \Sigma^\infty(L_{\A^1}(\mb{G}_{m,S}))\ar[d]\ar[r]& \Z_S(1)[1]\ar[d] \\
    \Sigma^\infty(p_*(L_{\A^1}(\mb{G}_{m,C}))) \ar[r,"\alpha"] &p_*\Z_C(1)[1]
\end{tikzcd}\]
commutes. Since $p^*$ and $L_{\A^1}$ commute, we have an exchange transformation $L_{\A^1}p_*\to p_*L_{\A^1}$. We claim that \[L_{\A^1}(p_*(\mb{G}_{m,C})) \to p_*(L_{\A^1}(\mb{G}_{m,C}))\]
is an equivalence. This can indeed be tested after tensoring with $\Q$ (this is compatible with $p_*$ by the same proof as \cite[Corollary~5.4.11]{em}) and with $\Z/\ell\Z$ for any prime number $\ell$, noting that $\mb{G}_{m,C}\otimes_\Z \Q$ is $\A^1$-local (see \textit{e.g.} the proof of \cite[Proposition 2.17]{plh}) and that $\mb{G}_{m,C}\otimes_\Z \Z/\ell\Z\simeq \mu_\ell[1]$ is $\A^1$-local over $S[1/\ell]$ and that $L_{\A^1}$ vanishes on $S\times_\Z \Z/\ell\Z$ by \cite[Proposition~A.3.4]{em}. 

As the map $\mb{G}_{m,S}\to \mathrm{H}^0p_*(\mb{G}_{m,C}))$ is an equivalence by \cite[Lemma~2.29]{plh}, we get a map $\beta \colon R^1p_*(\mb{G}_{m,C})[-1]\to P$ where $P$ is the cone of $\mb{G}_{m,S}\to p_*\mb{G}_{m,C}.$ By \cite[Section~9.3, Theorem~1]{neronmodels}, we have an exact sequence \[0\to \mathrm{Jac}(C/S)\to R^1p_*(\mb{G}_{m,C})\to \rho_\sharp(E)\to 0\] with $E$ an étale local system of rank $1$.
Applying $\Sigma^\infty L_{\A^1}$ thus yields a map \[\Sigma^\infty L_{\A^1}\mathrm{Jac}(C/S)\to K_2\simeq \Z_S\oplus K\to K.\] The lemma follows if we can show that it is an equivalence. This is true after tensoring with $\Q$ by \cite[Corollary~3.20]{plh} and after tensoring with $\Z/\ell\Z$ by classical computations of torsion étale cohomology (for instance, it can be tested on stalks where it is \cite[Corollaire~4.7]{sga4}).
\end{proof}

\section{The motivic t-structure on Ind-\texorpdfstring{$1$}{1}-motives}

We now define the motivic t-structure on Ind-$1$-motives. The following definition mimics \cite[Definition~4.10]{plh}. We will then compare this t-structure with the canonical t-structure of \cite[Corollary~C.5.5]{bvk} over a field. By showing how the t-structure interacts with torsion objects, namely that it induces the ordinary t-structure on them via the rigidity theorem \cite[Corollary~3.2]{bachmannrigidity}, we prove that the t-structure is compatible with pullbacks building on the case of rational coefficients from \cite[Theorem~4.1]{plh2}. This will allow us to show that the functor $\Phi_S$ sends Deligne $1$-motives with torsion to the heart of the motivic t-structure.
\begin{definition}
    Let $S$ be a scheme and let $\Lambda$ be a flat $\Z$-algebra. The \emph{ordinary motivic t-structure} (or \emph{motivic t-structure} for simplicity) is the t-structure on $\DM^{\mathrm{ind}1}_{\et}(S,\Lambda)$ generated by the objects 
    \[\mathcal{DG}_S=\{f_\sharp\Phi_X(M)\mid f\colon X \to S \text{ étale, } M \in \rmm_1^\D(S,\Lambda)\}.\]
\end{definition}

\begin{proposition}\label{JGgenerates}
    Let $S$ be an excellent scheme and let $\Lambda$ be a flat $\Z$-algebra. The motivic t-structure is generated by the following family:
    \[\mathcal{JG}_S=\{f_\sharp\Phi_X(M\otimes_\Z\Lambda)\mid f\colon X \to S \text{ étale, } M \in \{\Z_X,\mathbb{G}_{m,X}[-1]\} \sqcup \{\mathrm{Jac}(C/X)[-1]\}_{C/X \in \mathfrak{Curv}(X)} \}.\]
with $\mathfrak{Curv}(X)$ the family of all smooth projective curves over $X$ with geometrically connected fibres
and a section $X\to C$.
\end{proposition}
\begin{proof} This proof is analogous to \cite[Proposition~4.9]{plh} with \cite[Lemma~3.1.10]{AM1} as an extra ingredient.
We can assume $\Lambda=\Z$. Let $M\in \mathcal{DG}_S$, we have to prove that $M$ belongs to $\langle\mathcal{JG}_S\rangle_{-}$ where $\langle \ecal\rangle_{-}$ is the subcategory of $\DM_{\et}(S,\Lambda)$ generated from $\ecal$ by finite colimits, extensions and retracts. As $\mathcal{JG}_*$ is closed by $f_\sharp$ for $f$ étale, we may assume that $M$ is either of the form $\rho_!(L)$ or $\Sigma^\infty L_{\AAA}G[-1]$. The first case follows from \cite[Lemma~3.1.10]{AM1}. In the second case, we may further assume that $G$ is a torus or an abelian scheme. The case of a torus then follows from \Cref{toruscase} applying \cite[Lemma~3.1.10]{AM1} again. We can therefore assume $G$ to be an abelian scheme $A$. 

Now, by \Cref{i_*JG} below, we can assume $S$ to be reduced. Using the localization triangle \eqref{localization}and a noetherian induction as well as \Cref{i_*JG}, we are reduced to proving that $j^*(\Sigma^\infty L_{\AAA}(A))[-1]$ for $j\colon U\to S$ a non-empty open subscheme belongs to $\langle \mathcal{JG}_U\rangle_{-}$. But this is true generically by \cite[Theorem~11]{spacefillingcurveskatz} so it remains true on a non-empty open subscheme by \cite[Lemma~4.7(ii)]{plh}.
\end{proof}
\begin{lemma}\label{i_*JG}(\cite[Proposition~4.6]{plh}) 
    Let $i\colon Z\to S$ be a closed immersion with $S$ excellent. Then, $i_*(\mathcal{JG}_Z)\subseteq \langle\mathcal{JG}_S\rangle_{-}$ where $\langle \ecal\rangle_{-}$ is the subcategory of $\DM_{\et}(S,\Lambda)$ generated from $\ecal$ by finite colimits, extensions and retracts.
\end{lemma}
\begin{proof}
    The same proof as in \cite[Proposition~4.6]{plh} applies with $\Lambda$ instead of $\Q$ everywhere.
\end{proof}

\begin{proposition}\label{t-adj} Let $f$ be a quasi-finite morphism and let $g$ be a morphism. Then, the functors $f_!$ and $g^*$ are right t-exact with respect to the motivic t-structure.
\end{proposition}
\begin{proof} This follows from \Cref{i_*JG}.
\end{proof}

\begin{corollary}\label{t-adj2} Let $f$ be a morphism. 
\begin{enumerate} 
\item If $f$ is étale, the functor $f^*=f^!$ is t-exact.
\item If $f$ is finite, the functor $f_!=f_*$ is t-exact.
\end{enumerate}
\end{corollary}
Let us now describe the t-structure in the case of the spectrum of a field $k$. Firstly, the functor $L_\D$ behaves nicely:
\begin{constr}\label{Phiknatural}
    Let $k$ be field of characteristic exponent $p$.
    Then the functor \[L_\D\colon \rmm_1^\D(k,\Z[1/p])\to\D((\Sm_S)_\et,\Z[1/p])\] sends exact sequences to exact triangles.
Indeed, by \cite[Corollary~C.5.5]{bvk}, any exact sequence in $\rmm_1^\D(k,\Z[1/p])$ is equivalent to an exact sequence of complexes which is therefore sent to an exact triangle by $L_\D$.
This implies that the functor \[\Phi_k\colon \rmm_1^\D(k,\Z[1/p])\to \DM^1_{\et}(k,\Z[1/p])\]
sends exact sequences to exact triangles. Hence, by \cite[Corollary~7.4.12]{cisinski-bunke}, it extends uniquely to a functor \[\Phi_k^{\natural}\colon \D^b(\rmm_1^\D(k,\Z[1/p]))\to \DM_{\et}^1(k,\Z[1/p]).\]
This functor is an $\infty$-enhancement of the functor considered in \cite[Definition~2.7.1]{bvk} as they coincide on the heart and is therefore an equivalence by \cite[Theorem~2.1.2]{bvk} combined with Voevodsky's cancellation theorem for étale motives as in the proof of \Cref{curvedecomposition}.

If now $\Lambda$ is a localization of $\Z[1/p]$, tensoring with $\Perf_\Lambda$ yields an equivalence:
\[\Phi_k^\natural\colon \D^b(\rmm_1^\D(k,\Lambda))\to \DM_{\et}^1(k,\Lambda)\]
which extends the functor $\Phi_k$ from before.
\end{constr}

\begin{proposition}\label{bvkvsord}
    Let $k$ be field of characteristic exponent $p$ and let $\Lambda$ be a localization of $\Z[1/p]$. The functor \[\D^b(\rmm_1^\D(k,\Lambda))\to \DM_{\et}^{\mathrm{ind}1}(k,\Lambda)\] induced by $\Phi_k^\natural$ is a t-exact when the left hand side is endowed with its canonical t-structure and the right hand side is endowed with the motivic t-structure. In particular, the t-structure on $\DM_{\et}^{\mathrm{ind}1}(k,\Lambda)$ restricts to $\DM_{\et}^{1}(k,\Lambda)$ and $\Phi_k^\natural$ is a t-equivalence.
\end{proposition}
\begin{proof}
    If $\acal$ is a small abelian category, the t-structure on $\D^b(\acal)$ is generated by $\acal$. Hence it suffices to show that if $f\colon X\to \Spec(k)$ is étale (hence finite étale) and $M$ is a Deligne $1$-motives with torsion on $X$, $f_\sharp(\Phi_X(M))$ is a Deligne $1$-motive with torsion which is a consequence of \Cref{WeilRes} below. 
\end{proof}
\begin{lemma}\label{WeilRes}
    Let $g\colon T\to S$ be finite étale morphism of schemes. Denote by \[\mathrm{Res}_g\colon \Sh((\Sm_S)_\et,\Z)\to \Sh((\Sm_T)_\et,\Z)\] the Weil restriction functor. If $N=[\rho_\sharp(L)\xrightarrow{u} G]$ is a Deligne $1$-motive with torsion on $T$, then the complex $\mathrm{Res}_g(N):=[\rho_\sharp(\mathrm{Res}_g(L))\xrightarrow{\mathrm{Res}_g(u)} \mathrm{Res}_g(G)]$ is a Deligne $1$-motive with torsion. Furthermore, the diagram the diagram
\[\begin{tikzcd}
    \rmm_1^\D(T,\Z)\ar[r,"\Phi_T"]\ar[d,"\mathrm{Res}_g"]& \DM_{\et}(T,\Z)\ar[d,"f_\sharp"]\\
    \rmm_1^\D(S,\Z)\ar[r,"\Phi_S"]& \DM_{\et}(S,\Z)
\end{tikzcd}
\]
is commutative.
\end{lemma}
\begin{proof}
By \cite[Proposition~A.16]{plh}, the sheaf $\mathrm{Res}_g(G)$ is representable by a semi-abelian scheme. We claim that $\mathrm{Res}_g(L)$ is an étale local system; as étale morphism satisfy effectivity of descent for the étale topology, this can be checked étale locally so we can assume that the map $T\to S$ to be of the form $\bigsqcup\limits_{i=1}^n S\to S$ so that $\mathrm{Res}_g(L)=L^n$ which is again an étale local system. Hence, the complex $\mathrm{Res}_g(N):=[\rho_\sharp(\mathrm{Res}_g(L))\xrightarrow{\mathrm{Res}_g(u)} \mathrm{Res}_g(G)]$ is indeed a Deligne $1$-motive with torsion. 
The same proof as \cite[Proposition~Lemma 2.22.]{plh} then shows our diagram is commutative.     
\end{proof}

To handle the general case, we will need to understand better how the t-structure behaves with respect to torsion objects.
\begin{lemma}\label{torsioninduced}
    Let $S$ be a scheme and let $\ell$ be invertible on $S$. We let $\D_{\ell^\infty}(S_{\et},\Z)$ be the subcategory of $\D(S_{\et},\Z)$ made of those complexes $K$ such that $K\otimes_\Z \Z[1/\ell]=0$. 
    Then, the functor 
    \[\rho_!\colon \D_{\ell^\infty}(S_{\et},\Z)\to \DM^{\mathrm{ind}1}_{\et}(S,\Z)\] is t-exact when the left hand side is endowed with the ordinary t-structure (see \cite[Proposition~2.2.2]{AM1}) and the right hand side is endowed with the motivic t-structure.
\end{lemma}
\begin{proof}
    Note that the functor $\rho_!$ indeed lands in the subcategory of Ind-$1$-motives by \cite[Proposition~1.4.5]{AM1}. 
    Furthermore, the induced functor $\D(S_{\et},\Z)\to \DM^{\mathrm{ind}1}_{\et}(S,\Z)$ is right t-exact because the ordinary t-structure is generated by the representable sheaves by \cite[Proposition~2.2.1]{AM1}. 
    Hence it suffices to show that \[\rho_!\colon \D_{\ell^\infty}(S_{\et},\Z)\to \DM^{\mathrm{ind}1}_{\et}(S,\Z)\] is left t-exact, \textit{i.e.} sends non-negative objects to non-negative objects. Let $K$ be a non-negative complex in $\D_{\ell^\infty}(S_{\et},\Z)$. Since $K\otimes_\Z \Z[1/p]=0$, we get 
    \[K =\colim_n K\otimes_\Z \Z/p^n\Z[-1].\]
    As the objects of $\mc{JG}_S$ are geometric by \Cref{Delignemotare1mot}, if $M\in \mathcal{JG}_S$, knowing that $M$ is a geometric motive, we get that 
    \begin{align*}
        \map(M,\rho_!K)&=\colim_n \ \map(M,\rho_!(K\otimes_\Z\Z/p^n\Z[-1]))\\
        &=\colim_n\ \map(M\otimes_\Z \Z/p^n\Z,\rho_!(K\otimes_\Z \Z/p^n\Z[-1])).
    \end{align*}
    by \cite[Proposition~1.2.4(5)]{AM1}.
    Now, $M\otimes_\Z \Z/p^n\Z$ is of the form $\rho_!f_\sharp((\Z/p^n\Z)_X)$ or $\rho_!f_\sharp({}_{p^n}G)$ for $G=\mathbb{G}_{m,X}$ or an abelian $X$-scheme and so $M=\rho_!(L)$ where $L$ lies in the heart of $\D(S_{\et},\Z)$. As $K\otimes_\Z \Z/p^n\Z[-1]$ is non-negative, we get that $\map(M\otimes_\Z \Z/p^n\Z,\rho_!(K\otimes_\Z \Z/p^n\Z[-1]))$ is $(-1)$-connected and therefore, so is $\map(M,\rho_!K)$.
\end{proof}

\begin{definition}Let $S$ be a scheme. We define the category $\DM_{\et}^{\Q-1}(S,\Z)$ of \emph{rationally geometric $1$-motives} to be the subcategory of $\DM_{\et}(S,\Z)$ made of those motives $M$ such that $M\otimes_\Z \Q$ is a $1$-motive.
\end{definition}
\begin{proposition}\label{Q-cons} Let $S$ be an excellent scheme allowing resolution of singularities. The motivic t-structure on $\DM_{\et}^{\mathrm{ind}1}(S,\Z)$ induces a t-structure on the stable subcategory $\DM_{\et}^{\Q-1}(S,\Z)$. 
\end{proposition}
\begin{proof} Let $M$ be a rationally geometric $1$-motive. \cite[Lemma~1.1.5]{AM1} implies that $$\tau^{\leqslant 0}(M)\otimes_\Z \Q=\tau^{\leqslant 0}(M\otimes_\Z \Q).$$
But $\tau^{\leqslant 0}(M\otimes_\Z \Q)$ is a $1$-motive by \cite[Theorem~4.1(i)]{plh2}.
\end{proof}

\begin{proposition}\label{f^* ordinary} Let $S$ be an excellent scheme allowing resolution of singularities by alterations and let $f\colon T\rar S$ be a morphism of schemes.

Then, the functor $$f^*\colon \DM_{\et}^{\Q-1}(S,\Z)\to \DM_{\et}^{\Q-1}(T,\Z)$$ is t-exact when both sides are endowed with the motivic t-structure.
\end{proposition}
\begin{proof} The proof is the same as \cite[Proposition~4.2.5]{AM1}. We already know that the functor $f^*$ is right t-exact. Let $M$ be a rationally geometric $1$-motive which is t-non-negative. We have an exact triangle $$M\otimes_\Z \Q/\Z[-1]\rar M\rar M\otimes_\Z\Q .$$

By \cite[Lemma~1.1.5]{AM1}, the motive $M\otimes_\Z\Q$ is t-non-negative. Thus, the motive $M\otimes_\Z \Q/\Z[-1]$ is also t-non-negative.
Furthermore, we have an exact triangle $$f^*(M\otimes_\Z \Q/\Z[-1])\rar f^*(M)\rar f^*(M\otimes_\Z\Q).$$ By \cite[Theorem~4.1]{plh2}, the motive $f^*(M\otimes_\Z\Q)$ is t-non-negative. In addition, by \cite[Proposition~2.2.5]{AM1}, the motive $f^*(M\otimes_\Z \Q/\Z[-1])$ is t-non-negative. Hence, the motive $f^*(M)$ is also t-non-negative.
\end{proof}

\begin{corollary}\label{reducingtofields}
    Let $S$ be an excellent scheme allowing resolution of singularities by alterations and let $\Lambda$ be a localization of $\Z$.
Let $M$ be in $\DM^1_{\et}(S,\Lambda)$. Then, the following conditions are equivalent. \begin{enumerate}
        \item The $1$-motive $M$ is t-non-negative (resp. t-non-positive).
        \item For any $x$ in $S$, the $1$-motive $i_x^*(M)$ is t-non-negative (resp. t-non-positive).
        \item For any prime number $\ell$, the complex $\overline{\rho}_\ell(M)$ is t-non-negative (resp. t-non-positive).
    \end{enumerate}
\end{corollary}
\begin{rem}
    The same ideas would work over any flat $\Z$-algebra, provided that we prove that the method of \cite{plh2} works with $\Lambda\otimes_\Z \Q$-coefficients for any $\Lambda$. This is indeed true by direct inspection of all the necessary proofs in \cite{plh} which we invite the interested reader to check.
\end{rem}
\begin{proof}
First, if $M$ is bounded below, it has bounded below torsion in the sense of \Cref{conservativepts} below as $M\otimes_\Z \Z/p\Z$ is also bounded below with respect to the motivic t-structure and hence $\rho^!(M\otimes_\Z \Z/p\Z)$ is also bounded below with respect to the ordinary t-structure by \Cref{torsioninduced} and the rigidity theorem \cite[Corollary~3.2]{bachmannrigidity}.

Now, the family $(i_x^*)_{x\in S}$ is conservative bounded below torsion objects by \Cref{conservativepts}; \Cref{f^* ordinary} ensures that it is made of t-exact functors proving the first point. Proving the second point therefore reduces to the case of a point $S=\Spec(k)$. In that case, \Cref{bvkvsord} shows that the ordinary t-structure is bounded and therefore it suffices to show that the $\ell$-adic realization functor sends the heart to the heart. This is true because the Tate module of a Deligne motive with torsion $M$ coincides with $\rho_\ell\Phi_S(M)$ and the result thus follows from \Cref{Tatemodule}.
\end{proof}

\begin{lemma}\label{conservativepts}Let $S$ be a scheme. We denote by $\DM_{\et}^\mathrm{bbt}(S,\Z)$ the full subcategory of $\DM_{\et}(S,\Z)$ made of those motives $M$ such that each $M\otimes_\Z \Z/p\Z$ is bounded below when seen as a complex of étale sheaves through $\rho^!$ (which is an equivalence on $p$-torsion objects by the rigidity theorem \cite[Corollary~3.2]{bachmannrigidity}). We call it the category of \emph{étale motives with bounded below torsion}. 

The family formed by the
    \[\xi^* \colon \DM_{\et}^\mathrm{bbt}(S,\Z)\to \DM_{\et}^\mathrm{bbt}(\Omega,\Z)\] for $\xi\colon \Spec(\Omega)\to S$ a geometric point
    is conservative.
\end{lemma}
\begin{proof}
    Since the family $(-\otimes_\Z \Q,-\otimes_\Z \Z/p\Z \mid p\text{ prime})$ is conservative, it suffices to show the same lemma with coefficients $\Q$ of $\Z/p\Z$. The case of rational coefficients is \cite[Lemma~A.6]{ahplh} while the case of bounded below complexes of étale sheaves reduces to objects of the heart by a spectral sequence argument in which case it follows from \cite[Tag~03PU]{stacks}.
\end{proof}

\begin{corollary}\label{Deligne1motareinheart}
    Let $S$ be a scheme and let $\Lambda$ be flat $\Z$-algebra. Then the functor $\Phi_S$ has its essential image included in the heart of the motivic t-structure.
\end{corollary}
\begin{proof}
    We can assume that $\Lambda=\Z$. By \Cref{reducingtofields}, we can further assume $S$ to be the spectum of a field in which case it follows from \Cref{bvkvsord}.
\end{proof}

    
\begin{corollary}\label{bounded}
   Let $S$ be an excellent scheme allowing resolution of singularities by alterations and let $\Lambda$ be a localization of $\Z$. Then the objects of $\DM^1_{\et}(S,\Lambda)$ are bounded with respect to the motivic t-structure.
\end{corollary}
\begin{proof}
    We reduce to the the case of objects of the form $f_\sharp(\Lambda_X)$ for $X\to S$ smooth of relative dimension at most $1$. In that case, we claim that they lie in degrees ${0,1,2}$. To prove that, we can assume $S$ to be the spectrum of a field $k$ by \Cref{reducingtofields} which we can assume to be algebraically closed by \Cref{bvkvsord}. Now, take a smooth compactification $g\colon\overline{X}\to \Spec(k)$ of $X$. As $k$ is algebraically closed, the projective curve $\overline{X}$ has a rational point hence by \Cref{curvedecomposition}, the $1$-motive $g_\sharp(\Lambda_{\overline{X}})$ is in degrees ${0,1,2}$. If $i\colon Z\to \overline{X}$ is the closed complement of $X$, then $(ig)_\sharp(\Lambda_Z)$ is in degree $0$ by \Cref{t-adj2}. The result follows using the localization triangle \eqref{localization}. 
\end{proof}

\section{The motivic t-structure on \texorpdfstring{$1$}{1}-motives}
The goal of this section is to prove that the motivic t-structure restricts to $1$-motives. This will be a consequence of the following result which generalizes \cite[Lemma~6.12]{haas} by adapting the method of \cite[Theorem~4.2.7]{AM1}.
\begin{theorem}\label{DMsm1t}
Let $S$ be a regular scheme such that all the connected components of $S$ satisfy \Cref{etendable} and let $\Lambda$ be localization of $\Z'$. Then, $\Phi_S$ induces an equivalence:
    \[\Phi_S\colon\rmm_1^\D(S,\Lambda)\xrightarrow{\sim}\DM^{\sm1}_{\et}(S,\Lambda)\cap \DM^{\mathrm{ind}1}_{\et}(S,\Lambda)^\heart.\]
In particular, the motivic t-structure induces a t-structure on $\DM^{\sm1}_\et(S,\Lambda)$.
\end{theorem}
\begin{proof} We adapt the arguments of \cite[Lemma~6.12]{haas}, as we work with integral coefficients, we need to be more careful: we will use \Cref{embeddingDM1} below. We can assume $S$ to be connected. Let $\eta$ be its generic point. We have a commutative square 
\[\begin{tikzcd}
    \rmm_1^\D(S,\Lambda) \ar[hook,d,"\eta^*"]\ar[hook, "\Phi_S", r]
    & \DM^{\sm1}_{\et}(S,\Lambda)\cap \DM^{\mathrm{ind}1}_{\et}(S,\Lambda)^\heart \ar[d,"\eta^*"]\\
    \rmm_1^\D(\eta,\Lambda) \ar["\Phi_\eta", r]
    & \DM^{\sm1}_{\et}(\eta,\Lambda)\cap \DM^{\mathrm{ind}1}_{\et}(\eta,\Lambda)^\heart 
\end{tikzcd}\]
where $\Phi_S$ is fully faithful by \Cref{embeddingthm}, the left vertical arrow is fully faithful by \Cref{embeddinggeneric2} and $\Phi_\eta$ is an equivalence because in that case $\DM^{\sm1}_{\et}(\eta,\Lambda)=\DM^1_{\et}(\eta,\Lambda)$ and the motivic t-structure restricts to $1$-motives by \Cref{bvkvsord}. Now we claim that the functor 
\[\eta^*\colon\DM^{\sm1}_{\et}(S,\Lambda)\cap \DM^{\mathrm{ind}1}_{\et}(S,\Lambda)^\heart \to \DM^{\sm1}_{\et}(\eta,\Lambda)\cap \DM^{\mathrm{ind}1}_{\et}(\eta,\Lambda)^\heart \] is fully faithful. This follows from \Cref{embeddingDM1} using continuity \cite[Proposition~1.23]{plh} (note that it still works with integral coefficients using the same proof as \cite[Theorem~6.3.9]{em} replacing \cite[Théorème~8.10.5 \& Proposition~17.7.8]{ega4} with \cite[Lemma~1.24]{plh}).
Hence to show that the functor induced by $\Phi_S$ as above is an equivalence, we have to show that if $M$ belongs to $\DM^{\sm1}_{\et}(S,\Lambda)\cap \DM^{\mathrm{ind}1}_{\et}(S,\Lambda)^\heart$ then $N_\eta=\Phi_\eta^{-1}(\eta^*(M))$ has good reduction. By \Cref{bonnereduction}, this amounts to showing that if $\ell$ is a prime number distinct from the characteristic of $\eta$, $T_\ell(N_\eta)$ extends to a local system over $S[1/\ell]$. Now, $\overline{\rho}_\ell(M)$ is a dualizable object in $\D_\mathrm{cons}(S[1/\ell],\Z_\ell\otimes_\Z\Lambda)$ so it has perfect fibers by \cite[Theorem~4.13]{hrs} and is in the heart of the t-structure so it belongs to $\Loc_{S[1/\ell]}(\Z_\ell\otimes_\Z \Lambda)$ by \cite[Theorem~6.2]{hrs} and by definition, its fiber to $\eta$ is $T_\ell(N_\eta)=\overline{\rho}_\ell(\eta^*(M))$.

The fact that the motivic t-structure induces a t-structure on $\DM^{\sm1}_\et(S,\Lambda)$ follows from \cite[Lemma~1.2.3]{aps}: indeed knowing that the objects of $\DM^{\sm1}_\et(S,\Lambda)$ are bounded by \Cref{bounded}, we only have to show that if $f\colon M\to N$ is a map of Deligne $1$-motives, then $\ker(f)$ is still a Deligne $1$-motive. This is true because $\rmm_1^\D(S,\Lambda)$ is an abelian category by \Cref{abeliancat} and the kernel computed there is the same as the kernel computed in $\DM^{\mathrm{ind}1}_{\et}(S,\Lambda)^\heart$ because the functor between those abelian categories is exact and fully faithful; its exactness follows from the fact that the functor $\Phi_S$ sends short exact sequences to exact triangles which are therefore again short exact sequences. This is true because for any $x\in S$, the functor $\Phi_x$ sends exact sequences to exact sequences (see \Cref{Phiknatural}) and the family of the \[i_x^*\colon \DM^{\Q-1}_{\et}(S,\Lambda)^\heart\to \DM^{\Q-1}_{\et}(x,\Lambda)^\heart\]
is exact and conservative by \Cref{conservativepts} and \Cref{f^* ordinary} (note that being bounded for the ordinary t-structure implies having bounded below torsion because of \Cref{torsioninduced}).
\end{proof}

\begin{lemma}\label{embeddingthm}
Let $S$ be a regular scheme let $\Lambda$ be a flat $\Z'$-algebra. Then the functor \[\Phi_S\colon \rmm_1^\D(S,\Lambda)\to\DM^{\sm1}_{\et}(S,\Lambda)\] is fully faithful.
\end{lemma}
\begin{proof}
    We can assume that $S$ is connected with generic point $\eta$ and that $\Lambda=\Z'$. In that case, the functor $\Phi_\eta$ is fully faithful by \cite[Theorem~2.1.2]{bvk} combined with Cisinski-Déglise's version of Voevodsky's cancellation theorem for étale motives (see the proof of \Cref{curvedecomposition}). Let $M$ and $N$ be in $\rmm_1^\D(S,\Z')$. Then, as by \Cref{f^*}, the functor $\Phi_?$ is compatible with pullbacks, we have a commutative diagram
    \[\begin{tikzcd}
        \Hom_{\rmm_1^\D(S,\Z')}(M,N) \ar[d]\ar[r]& \Hom_{\DM_{\et}(S,\Z')}(\Phi_S(M),\Phi_S(N)) \ar[d]\\
        \Hom_{\rmm_1^\D(\eta,\Z')}(M_\eta,N_\eta) \ar[r]& \Hom_{\DM_{\et}(\eta,\Z')}(\Phi_\eta(M_\eta),\Phi_\eta(N_\eta)).
    \end{tikzcd}\]
    The left vertical arrow is an equivalence by \Cref{embeddinggeneric2} while we just proved that the bottom horizontal one is an equivalence. Hence, it suffices to show that the map \[\Hom_{\DM_{\et}(S,\Z')}(\Phi_S(M),\Phi_S(N)) \to\Hom_{\DM_{\et}(\eta,\Z')}(\Phi_\eta(M_\eta),\Phi_\eta(N_\eta))\]
    is injective. This follows from \Cref{embeddingDM1} below and continuity. 
\end{proof}

\begin{lemma}\label{embeddingDM1}
    Let $S$ be a regular scheme and let $j\colon U\to S$ be a dense open immersion. Then, the functor 
    \[j^*\colon \DM^{\sm1}_{\et}(S,\Lambda)\cap \DM^{\mathrm{ind}1}_{\et}(S,\Lambda)^\heart\to \DM^{\sm1}_{\et}(U,\Lambda)\cap \DM^{\mathrm{ind}1}_{\et}(U,\Lambda)^\heart\]
    is fully faithful.
\end{lemma}
\begin{proof}
We must first recall some notions from \cite{plh}, for $T$ a scheme, we have
\begin{enumerate}
\item the category $\DM_{\et}^{\mathrm{ind}\mathrm{coh}}(T,\Lambda)$ of Ind-cohomological motives to be the localizing subcategory of $\DM_{\et}(T,\Lambda)$ generated by the $f_*\Lambda_X$ for $f\colon X \to T$ proper.
\item the category $\DM_{\et}^{\mathrm{ind}\mathrm{coh}-1}(T,\Lambda)$ of Ind-cohomological $1$-motives to be the localizing subcategory of $\DM_{\et}(T,\Lambda)$ generated by the $f_*\Lambda_X$ for $f\colon X \to T$ proper of relative dimension at most $1$.
\item The inclusion $\DM_{\et}^{\mathrm{ind}\mathrm{coh}-1}(T,\Lambda)\to \DM_{\et}^{\mathrm{ind}\mathrm{coh}}(T,\Lambda)$ has a right adjoint $\omega^1$.
\end{enumerate}
Note that by \cite[Proposition~1.28]{plh}, we have a t-equivalence of categories
\[(-1)\colon \DM^{\mathrm{ind}1}_{\et}(T,\Lambda)\to \DM_{\et}^{\mathrm{ind}\mathrm{coh}-1}(T,\Lambda)\]
with the t-structure on the right hand side generated by $\mc{JG}_T(-1)$.

We now claim that if $M$ belongs to $\DM^{\sm1}_{\et}(S,\Lambda)\cap \DM^{\mathrm{ind}1}_{\et}(S,\Lambda)^\heart$, then the natural map
\[M(-1)\to \tau^{\leq 0}(\omega^1(j_*j^*M(-1)))\] is an equivalence which would prove the Lemma. We let $N=M(-1)$ and we let $i\colon Z\to S$ be the reduced closed complement of $S$. As $\omega^1$ commutes with $i_*$, the localization triangle \eqref{colocalization} yields an exact triangle 
\[i_*\omega^1i^!(N)\to N\to \omega^1j_*j^*N.\]
hence it suffices to prove that $i_*\omega^1i^!(N)$ lies in degree at least $2$ with respect to the motivic t-structure. For this we adapt \cite[Lemma~4.7]{plh2} to the setting of integral coefficients: chose a stratification $\emptyset =\overline{Z}_{m+1} \subseteq \overline{Z}_m\subseteq \cdots \subseteq \overline{Z}_0=Z$  by closed subsets with
$Z_s := \overline{Z}_s \setminus \overline{Z}_{s+1}$ regular and equidimensional. Write $c_s$ for the codimension of $Z_s$ in $S$, since $U$ is dense, we have $c_s>0$. Write $Z_i\xrightarrow{j_s} \overline{Z}_i \xleftarrow{k_s}\overline{Z}_{s+1}$ and let $\theta_s\colon \overline{Z}_{s+1}\to S$ be the closed immesion with $\theta_{-1}=i$. By localisation, we have distinguished
triangles
\[(k_s)_*\omega^1\theta_s^!(N)\to \omega^1\theta_{s-1}^!(N)\to \omega^1 (j_s)_*j_s^*\omega^1\theta_{s-1}^!(N)\]
using the compatibilities of $\omega^1$ with the six functors described in \cite[Proposition~3.3]{plh}.
By absolute purity \cite[Theorem~4.6.1]{em}, using the fact that $N$ is dualizable, we have $\theta_{s-1}^!(N)=\theta_s^*(N)(-c_s)[-2c_s]$. As $j_s^*$ is right t-exact by \Cref{t-adj}, the functor $(j_s)_*$ is left t-exact \textit{i.e.} preserves non-positive objects. Hence, combining the exactness properties of \Cref{t-adj2} and \Cref{f^* ordinary}, it suffices to show that $\omega^1(\theta_{s-1}^*(N)(-c_s)[-2c_s])$ lies in degree at least $2$. 

Now write the exact triangle 
\[\omega^1(\theta_{s-1}^*(N)(-c_s)[-2c_s])\to\omega^1(\theta_{s-1}^*(N)(-c_s)[-2c_s])\otimes_\Z \Q \to \omega^1(\theta_{s-1}^*(N)(-c_s)[-2c_s])\otimes_\Z \Q/\Z. \]
The same proof as \cite[Proposition~2.1.5]{AM2} yields that $\omega^1$ is compatible with tensoring with $\Q$ and $\Z/n\Z$ for any $n$. Furthermore, the functor $\omega^1$ induces the identity on torsion motives as they are $0$-motives (see \cite[Proposition~2.1.6]{AM2}). Hence, we get an exact triangle:
\[\omega^1(\theta_{s-1}^*(N)(-c_s)[-2c_s])\to\omega^1(\theta_{s-1}^*(N\otimes_\Z \Q)(-c_s))[-2c_s] \to \colim_n\ \theta_{s-1}^*(N\otimes_\Z \Z/n\Z)(-c_s)[-2c_s]). \]
Now as $U$ is dense, the term in the middle lies in degree at least $2$ by \cite[Lemma~4.7]{plh}, while each $\theta_{s-1}^*(N\otimes_\Z \Z/n\Z)(-c_s)[-2c_s]$ lies in degree at least $1$ because $N\otimes_\Z \Z/n\Z$ is in degree at least $-1$ and $c_s>0$. 
Hence, the same argument as at the end of the proof of \Cref{torsioninduced} yields that their colimit is also in degree at least $1$. This shows that $\omega^1(\theta_{s-1}^*(N)(-c_s)[-2c_s])$ is in degree at least $2$ finishing the proof.
\end{proof}

\begin{corollary}
\label{ordinary is trop bien}  Let $S$ be a $\Q$-scheme or a Dedekind scheme and
let $\Lambda$ localization of $\Z$. Then the motivic t-structure induces a t-structure on $\DM^1(S,\Lambda)$. This t-structure is furthermore compatible with pullback and with the $\ell$-adic realization functors.
\end{corollary}
\begin{proof} Once we have proved that the t-structure restricts to $1$-motives, the rest will follow from \Cref{reducingtofields}.
We adapt the proof of \cite[Theorem~4.2.7]{AM1}.
If $\pp$ is a prime ideal of $\Z$ and if $X$ is a scheme, we say that an Ind-$1$-motive is \emph{$\pp$-geometric} when its image in the category $\DM^{\mathrm{ind}1}_{\et}(X,\Lambda)\otimes_{\Perf_{\Z}} \Perf_{\Z_\pp}$ belongs to $\DM^{1}_{\et}(X,\Lambda)\otimes_{\Perf_{\Z}} \Perf_{\Z_\pp}$. 
Let $M$ be a $1$-motive over $S$. We want to show that the Ind-$1$-motive $\tau^{\leqslant 0}(M)$ is a $1$-motive. By \cite[Proposition~1.1.10]{AM1}, it suffices to show that it is $\pp$-geometric for any maximal ideal $\pp$ of $\Z$. 

Let $p$ be a generator of $\pp$, exactly as in the proof of \cite[Theorem~4.2.7]{AM1}, we can reduce to the case when 
 \begin{enumerate}
     \item $p$ is invertible on $S$ or $S$ is of characteristic $p$.
     \item $M$ belongs to $\DM^{\sm1}(S,\Lambda)$.
 \end{enumerate}

Assume that $S$ is of characteristic $p$. It suffices to show that the Ind-$1$-motive $\tau^{\leqslant 0}(M)$ is a $1$-motive. Using \cite[Proposition~A.3.4]{em}, we can assume that $p$ is invertible in $\Lambda$. The result then follows from \Cref{DMsm1t} noting that \Cref{etendable} holds because of \Cref{whendoesetendablehold}.

If $p$ is invertible on $S$, note that the motivic t-structure induces a t-structure on \[\DM^{\mathrm{ind}1}_{\et}(X,\Lambda)\otimes_{\Perf_{\Z}} \Perf_{\Z_\pp}\] using \cite[Proposition~1.1.8]{AM1}. It suffices to show that it induces a t-structure on 
\[\DM^{\sm1}_{\et}(X,\Lambda)\otimes_{\Perf_{\Z}} \Perf_{\Z_\pp}\xrightarrow{\sim}\DM^{\sm1}_{\et}(X,\Lambda\otimes_\Z\Z_\pp).\] The latter has a t-structure induced by the motivic t-structure by \Cref{DMsm1t} noting that \Cref{etendable} holds because of \Cref{whendoesetendablehold}. Hence it suffices to see that the functor 
\[\DM^{\sm1}_{\et}(X,\Lambda\otimes_\Z\Z_\pp)\to \DM^{\mathrm{ind}1}_{\et}(X,\Lambda)\otimes_{\Perf_{\Z}} \Perf_{\Z_\pp}\] is t-exact. To that end it suffices to show that it sends $\rmm_1^\D(S,\Lambda\otimes_\Z\Z_\pp)$ to the heart which is true because the functor $\Phi_S$ sends Deligne $1$-motives to the heart of the motivic t-structure by \Cref{Deligne1motareinheart}.
\end{proof}

\section*{Acknowledgments}
I am very grateful to Frédéric Déglise and Fabrizio Andreatta for their constant support. I would like to thank Swann Tubach warmly for many comments on earlier versions of this paper. I also thank Luca Barbieri-Viale, Federico Binda, Mattia Cavicchi, Daniel Kriez and Alberto Vezzani with whom I had helpful discussions.

I was funded by the Prin 2022: The arithmetic of
motives and L-functions and by the ANR HQDIAG.

	\bibliographystyle{alpha}
	\bibliography{biblio.bib}

\end{document}